\documentclass[11pt]{amsart}
\pdfoutput=1

\pdfpageattr{/Group <</S /Transparency /I true /CS /DeviceRGB>>}

\usepackage[T1]{fontenc}
\usepackage[english]{babel}
\usepackage{geometry}
\usepackage{amsfonts,amsmath,amssymb,amsthm,mathrsfs}
\usepackage{tikz}
\usetikzlibrary{calc}
\usepackage{caption}
\usepackage{subcaption}
\usepackage{xspace}
\usepackage{booktabs,tabularx}
\usepackage{dcolumn}
\usepackage{marvosym}
\usepackage{microtype}
\usepackage{enumerate}
\usepackage{mathtools}
\usepackage{hyperref}

\microtypesetup{tracking,kerning,spacing}
\microtypecontext{spacing=nonfrench}

\hypersetup{colorlinks,linkcolor=blue ,citecolor=blue, urlcolor=blue}

\newcommand\cprime{$'$}

\newcommand\polymake{{\tt polymake}\xspace}

\newcommand\NN{{\mathbb N}}
\newcommand\LL{{\mathbb L}}
\newcommand\KK{{\mathbb K}}
\newcommand\KKt{{\mathbb K}\{\!\!\{t\}\!\!\}}

\newcommand\RR{{\mathbb R}}

\newcommand\cL{\mathcal{L}}

\newcommand\SetOf[2]{\left\{\left.#1\vphantom{#2}\ \right|\ #2\vphantom{#1}\right\}}

\newcommand\transpose[1]{{#1}^{\top}}
\DeclareMathOperator{\val}{val}

\DeclareMathOperator{\Gr}{Gr} 
\DeclareMathOperator{\TGr}{TGr} 
\DeclareMathOperator{\characteristic}{char} 
\DeclareMathOperator{\Dr}{Dr} 

\DeclareMathOperator{\rank}{rk}

\DeclareMathOperator{\size}{\#}
\newcommand{\sfLift}{\Lambda}

\newcommand{\snow}{\mathcal{S}}

\renewcommand{\epsilon}{\varepsilon}

\theoremstyle{plain}
    \newtheorem{theorem}{Theorem}
    \newtheorem{corollary}[theorem]{Corollary}
    \newtheorem{lemma}[theorem]{Lemma}
    \newtheorem{proposition}[theorem]{Proposition}
\theoremstyle{definition}
    \newtheorem{remark}[theorem]{Remark}
    \newtheorem{example}[theorem]{Example}
    \newtheorem{definition}[theorem]{Definition}

    \newtheorem{question}{Question}

\title{Matroids from hypersimplex splits}

\author{Michael Joswig \and Benjamin Schr\"oter}

\address[Michael Joswig, Benjamin Schr\"oter]{
  Institut f{\"u}r Mathematik,
  TU Berlin,
  Str.\ des 17. Juni 136, 10623 Berlin, Germany
}
\email{\{joswig,schroeter\}@math.tu-berlin.de}

\thanks{
Research by M. Joswig is carried out in the framework of Matheon supported by Einstein Foundation Berlin. 
Further support by Deutsche Forschungsgemeinschaft (SFB-TRR 109: ``Discretization in Geometry and Dynamics'' and SFB-TRR 195: ``Symbolic Tools in Mathematics and their Application'')
}


\begin{document}

\begin{abstract}
  A class of matroids is introduced which is very large as it strictly contains all paving matroids as special cases.
  As their key feature these \emph{split matroids} can be studied via techniques from polyhedral geometry.
  It turns out that the structural properties of the split matroids can be exploited to obtain new results in tropical geometry, especially on the rays of the tropical Grassmannians.
\end{abstract}
\maketitle

\section{Introduction}
\noindent
The purpose of this paper is to introduce, to characterize and to exploit a new class of matroids, which we call \emph{split matroids}.
We will argue that there are good reasons to study these matroids for the sake of matroid theory itself.
Additionally, however, they also give rise to a large and interesting class of tropical linear spaces.
In this way we can use split matroids to answer some questions which previously arose in the investigation of tropical Grassmannians \cite{SpeyerSturmfels:2004} and Dressians \cite{HerrmannJensenJoswigSturmfels:2009,HerrmannJoswigSpeyer:2012}.

The split matroids are motivated via polyhedral geometry in the following way.
For a given matroid $M$ the convex hull of the characteristic vectors of the bases of $M$ is the \emph{matroid polytope} $P(M)$.
The hypersimplices $\Delta(d,n)$ are the matroid polytopes corresponding to the uniform matroids $U_{d,n}$.
If $M$ has rank $d$ and $n$ elements, the matroid polytope $P(M)$ is a subpolytope of $\Delta(d,n)$.
Studying matroids in polyhedral terms goes back to Edmonds~\cite{Edmonds:1970}.

A \emph{split} of a polytope is a subdivision with precisely two maximal cells.
These subdivisions are necessarily regular, and the cells are matroid polytopes.
The hyperplane spanned by the intersection of the two maximal cells is the corresponding \emph{split hyperplane}.
Clearly this hyperplane determines the split, and it yields a facet of both maximal cells.
As our first contribution we show the following converse.
Each facet of a matroid polytope $P(M)$ corresponds to either a hypersimplex facet or a hypersimplex split (Proposition~\ref{prop:mp-facets}).
We call the latter the \emph{split facets} of $P(M)$.
The hypersimplex facets correspond to matroid deletions and contractions, and the hypersimplex splits have been classified in \cite{HerrmannJoswig:2008}.
Now the matroid $M$ is a \emph{split matroid} if the split facets of $P(M)$ satisfy a compatibility condition.
We believe that these matroids are interesting since they form a large class but feature stronger combinatorial properties than general matroids.
``Large'' means that they comprise the paving matroids and their duals as special cases (Theorem~\ref{thm:paving_matroids}).
It is conjectured that asymptotically almost all matroids are paving matroids \cite{MayhewNewmanWelshWhittle:2011} and \cite[15.5.8]{Oxley:2011}.
In particular, this would imply that almost all matroids are split.
Section~\ref{app:statistics} in the appendix provides statistical data based on a census of small matroids which has been obtained by Matsumoto, Moriyama, Imai and Bremner \cite{Matsumoto:2012}.

We characterize the split matroids in terms of deletions and contractions, i.e., in pure matroid language (Theorem~\ref{thm:split-matroid} and Proposition~\ref{prop:disconnected}).
This way it becomes apparent that the basic concepts of matroid splits and split matroids make several appearances in the matroid literature.
For instance, a known characterization of paving matroids implicitly makes use of this technique; see \cite[Prop. 2.1.24]{Oxley:2011}.
Splits also occur in a recent matroid realizability result by Chatelain and Ram\'{\i}rez Alfons\'{\i}n \cite{ChatelainRamirez:2014}.
Yet, to the best of our knowledge, so far split matroids have not been recognized as an interesting class of matroids in their own right.

One motivation to study matroid polytopes comes from tropical geometry; see Maclagan and Sturmfels \cite{MaclaganSturmfels:2015}.
Tropical geometry is related to the study of an algebraic variety defined over some field with a discrete valuation, and a \emph{tropical variety} is the image of such a variety under the valuation map.
In particular, a \emph{tropical linear space} corresponds to a polytopal subdivision of the hypersimplices where each cell is a matroid polytope; see De~Loera, Rambau and Santos \cite{LoeraRambauSantos:2010} for general background on subdivisions of polytopes.
The \emph{Dressian} $\Dr(d,n)$ is the polyhedral fan of lifting functions for the (regular) matroid subdivisions of $\Delta(d,n)$.
By definition this is a subfan of the secondary fan.
In general, $\Dr(d,n)$ has maximal cones of various dimensions, i.e., it is not pure.
In work of Dress and Wenzel \cite{DressWenzel:1992} these lifting functions occur as ``valuated matroids''.
Using split matroids we provide exact asymptotic bounds for $\dim \Dr(d,n)$ (Theorem~\ref{thm:dim}).

A tropical linear space is \emph{realizable} if it arises as the tropicalization of a classical linear space.
It is known from work of Speyer \cite{Speyer:2005,Speyer:2009} that the realizability of tropical linear spaces is related with the realizability of matroids.
Here we give a first characterization of matroid realizability in terms of certain tropical linear spaces (Theorem \ref{thm:realizable}).
The subset of $\Dr(d,n)$ which corresponds to the realizable tropical linear spaces is the \emph{tropical Grassmannian}.
The latter is also equipped with a fan structure, which is inherited from the Gr\"obner fan of the $(d,n)$-Pl\"ucker ideal.
Yet it is still quite unclear how these two fan structures are related.
Here we obtain a new structural result by showing that, via split matroids, one can construct very many non-realizable tropical linear spaces which correspond to rays of the Dressian (Theorem~\ref{thm:rays}).
It was previously unknown if \emph{any} such ray exists.
The Dressian rays correspond to those tropical linear spaces which are most degenerate.
Once they are known it is ``only'' necessary to determine the common refinements among them to describe the entire Dressians.
In this way the rays yield a condensed form of encoding.
It is worth noting that the Dressians have far fewer rays than maximal cones.
For instance, $\Dr(3,8)$ has $4 748$ maximal cones but only twelve rays, up to symmetry \cite[Theorem~31]{HerrmannJoswigSpeyer:2012}.

\section{Matroid polytopes and their facets}
\noindent
Throughout this paper let $M$ be a matroid of rank $d$ with ground set $[n]=\{1,2,\dots,n\}$.
Frequently, we use the term \emph{$(d,n)$-matroid} in this situation.
We quickly browse through the basic definitions; further details about matroid theory can be found in the books of Oxley \cite{Oxley:2011} and White \cite{White:1986}.
We use the notation of Oxley \cite{Oxley:2011} for specific matroids and operations.
The matroid $M$ is defined by its \emph{bases}.
They are $d$-element subsets of $[n]$ which satisfy an abstract version of the basis exchange condition from linear algebra.
Subsets of bases are called \emph{independent}, and a dependent set which is minimal with respect to inclusion is a \emph{circuit}.
An element $e\in[n]$ is a \emph{loop} if it is not contained in any basis, and it is a \emph{coloop} if it is contained in all the bases.
Let $S$ be a subset of $[n]$.
Its \emph{rank}, denoted by $\rank(S)$, is the maximal size of an independent set contained in $S$.
The set $S$ is a \emph{flat} if for all $e\in[n]-S$ we have $\rank(S+e)=\rank(S)+1$.
The entire ground set and, in the case of loop-freeness, also the empty set are flats; the other flats are called \emph{proper flats}.
The set of flats of $M$, partially ordered by inclusion, forms a geometric lattice, the \emph{lattice of flats}.
The matroid $M$ is \emph{connected} if there is no \emph{separator set} $S\subsetneq[n]$ with $\rank(S)+\rank([n]-S)=d$.
A connected matroid with at least two elements does not have any loops or coloops.
A disconnected $(d,n)$-matroid decomposes in a \emph{direct sum} of an $(r,m)$-matroid $M'$ and a rank $d-r$ matroid $M''$ on $\{m+1,\ldots,n\}$, i.e, a basis is the union of a basis of $M$ and a basis of~$N$.
We write $M'\oplus M''$ for the direct sum.

For a flat $F$ of rank $r$ we define the \emph{restriction} $M|F$ of $F$ with respect to $M$ as the matroid on the
ground set $F$ whose bases are the sets in the collection
\[
\SetOf{\sigma\cap F}{\sigma \text{ basis of } M \text{ and } \#(\sigma\cap F)=r} \enspace .
\]
Dually, the \emph{contraction} $M/F$ of $F$ with respect to $M$ is the matroid on the ground set $[n]-F$ whose
bases are given by
\[
\SetOf{\sigma\setminus F}{\sigma \text{ basis of } M \text{ and } \#(\sigma\cap F)=r} \enspace .
\]
The restriction $M|F$ is a matroid of rank $r$, while the contraction $M/F$ is a matroid on the complement of
rank $d-r$.

Via its characteristic function on the elements, a basis of $M$ can be read as a $0/1$-vector of length $n$ with exactly $d$ ones.
The joint convex hull of all such points in $\RR^n$ is the \emph{matroid polytope} $P(M)$ of $M$.
A basic reference to polytope theory is Ziegler's book \cite{Ziegler:2000}.
It is immediate that the matroid polytope of any $(d,n)$-matroid is contained in the $(n{-}1)$-dimensional simplex
\[
\Delta \ = \ \SetOf{x\in\RR^n}{x_1\geq 0,\, x_2\geq 0,\,\ldots,\, x_n\geq 0, \ \sum^n_{i=1} x_i=d} \enspace .
\]
Combinatorial properties of $M$ directly translate into geometric properties of $P(M)$ and vice versa.
For instance, Edmonds~\cite[(8) and (9)]{Edmonds:1970} gave the exterior description
\begin{equation}\label{eq:matroid-exterior}
  P(M) \ = \ \SetOf{x\in\Delta}{\sum_{i\in F}x_i\leq\rank(F), \text{ where $F$ ranges over the set of flats}}
\end{equation}
of the matroid polytope $P(M)$ in terms of the flats.
The set
\[
P_M(F) \ := \ \SetOf{x\in P(M)}{\sum_{i\in F}x_i=\rank(F)}
\]
is the face of $P(M)$ defined by the flat $F$.  Clearly, some flats lead to redundant inequalities.
A \emph{flacet} of $M$ is a flat which defines a facet of $P(M)$ and which is minimal with respect to inclusion among all flats that define the same facet.
They have been characterized in purely combinatorial terms by Fujishige \cite[Theorems 3.2 and~3.4]{Fujishige:1984} and, independently, by Feichtner and Sturmfels \cite[Propositions 2.4 and~2.6]{FeichtnerSturmfels:2005} as follows.
\begin{proposition}\label{prop:matroid-polytope}\strut
  \begin{enumerate}
  \item The dimension of $P(M)$ equals $n$ minus the number of connected components of $M$.
  \item\label{it:matroid-polytope:flacet} A proper flat $F$ whose restriction $M|F$ and contraction $M/F$ both are connected is a flacet of $M$.
  \item\label{it:matroid-polytope:product} For each proper flat $F$ we have $P_M(F) = P(M|F)\times P(M/F) = P(M|F \oplus M/F)$.
  \end{enumerate}
\end{proposition}

\begin{remark}\label{rem:flacets}
   Proposition~\ref{prop:matroid-polytope}(\ref{it:matroid-polytope:flacet}) characterizes the flacets of a connected matroid.
   For a disconnected matroid the notion of a flacet is somewhat subtle.
   First, in the disconnected case there are proper hyperplanes which contain the entire matroid polytope.
   Such a hyperplane is not facet defining and the corresponding flat is not a flacet.
   Second, for any given facet the defining inequality is never unique.
   In our definition we choose a specific representative by picking the inclusion minimal flat.
   If a flat is a direct sum $F\oplus G$ then $P_M(F\oplus G)$ is the intersection of the two faces $P_M(F)$ and $P_M(G)$.
   In particular, the restriction to a flacet is always connected, while the contraction is not.
\end{remark}

The \emph{hypersimplex} $\Delta(d,n)$ is the matroid polytope of the uniform matroid $U_{d,n}$ of rank $d$ on $n$
elements.  Its vertices are all the $0/1$-vectors of length $n$ with exactly $d$ ones.  As $\Delta(d,n)$ is the
intersection of the unit cube $[0,1]^n$ with the hyperplane $\sum x_i=d$, the $2n$ facets of $[0,1]^n$ give rise to a
facet description for $\Delta(d,n)$.  In this case the flacets are the $n$ flats with one element.  The matroid polytope
of any $(d,n)$-matroid is a subpolytope of $\Delta(d,n)$.  The following converse, obtained by Gel{\cprime}fand,
Goresky, MacPherson and Serganova, is a fundamental characterization.  The vertex-edge graph of the $(d,n)$-hypersimplex
is called the \emph{Johnson graph} $J(d,n)$.  This is a $[d\cdot(n-d)]$-regular undirected graph with $\tbinom{n}{d}$
nodes; each of its edges corresponds to the exchange of two bits.

\begin{proposition}[{\cite[Theorem~4.1]{GelfandEtAl:1987}}]\label{prop:gelfand}
  A subpolytope $P$ of $\Delta(d,n)$ is a matroid polytope if and only if the vertex-edge graph of $P$ is a subgraph of the Johnson graph $J(d,n)$.
\end{proposition}

In the subsequent sections we will be interested in polytopal subdivisions of hypersimplices and, more generally, arbitrary matroid polytopes.
The following concept is at the heart of our deliberations.
A \emph{split} of a polytope $P$ is a polytopal subdivision $\Sigma$ of $P$ with exactly two maximal cells.
The two maximal cells share a common codimension-$1$-cell, and its affine span is the \emph{split hyperplane} of $\Sigma$.

\begin{proposition}[{\cite[Lemma~5.1]{HerrmannJoswig:2008}}]\label{prop:splits_of_the_hypersimplex}
  For any proper non-empty subset $S\subsetneq[n]$ and any positive integer $\mu<d$ with $d-\#S<\mu<n-\#S$ the \emph{$(S,\mu)$-hyperplane} equation
  \begin{equation}\label{eq:split}
    \mu\sum_{i\in S} x_i \ = \ (d-\mu) \sum_{j\not\in S} x_j
  \end{equation}
  defines a split of $\Delta(d,n)$.
  Conversely, each split of $\Delta(d,n)$ arises in this way.
\end{proposition}

The split equation above is given in its homogeneous form. Since the hypersimplices are not full-dimensional this can be
rewritten in many ways.  For instance, taking $\sum_i x_i =d$ into account yields the inhomogeneous equation
\begin{equation}\label{eq:inhomogeneous}
  \sum_{i\in S} x_i \ = \ d-\mu \enspace ,
\end{equation}
which is equivalent to \eqref{eq:split}.
Note that the equation \eqref{eq:inhomogeneous} has a similar shape as the inequalities in the exterior description \eqref{eq:matroid-exterior} of the matroid polytopes.
A direct computation shows that the intersection of $\Delta(d,n)$ with the $(S,\mu)$-hyperplane is the product of hypersimplices
\begin{equation}\label{eq:product}
  \Delta(d-\mu,S)\times\Delta(\mu,[n]-S) \enspace ,
\end{equation}
where we use a complementary pair of subsets of $[n]$ (instead of cardinalities) in the second arguments of the
hypersimplex notation to fix
 the embedding into $\Delta(d,n)$ as a subpolytope.

\begin{remark}\label{rem:splits-of-matroid-polytopes}
  By \cite[Observation~3.1]{HerrmannJoswig:2008} a hyperplane $H$ which separates an arbitrary polytope $P$ defines a split of $P$ if and only if $H$ does not intersect any edge of $P$ in its relative interior:
  Clearly, if $H$ separates any edge of $P$ it does not define a subdivision of $P$ without new vertices.
  Conversely, if no edge of $P$ gets separated then $H$ induces a split with the two maximal cells $P\cap H^+$ and $P\cap H^-$, where $H^+$ and $H^-$ are the two affine halfspaces defined by $H$.
   In view of Proposition~\ref{prop:gelfand} we conclude that the (maximal) cells of any split of a hypersimplex form matroid polytopes.
  See also \cite[Proposition 3.4]{HerrmannJensenJoswigSturmfels:2009}. 
\end{remark}

We want to express Proposition~\ref{prop:splits_of_the_hypersimplex} in terms of matroids and their flats.

\begin{lemma}\label{lem:flat-split}
  Let $F$ be a proper flat such that $0<\rank(F) < \#F$.
  If there is an element $e$ in $[n]-F$ which is not a coloop then the $(F,d-\rank(F))$-hyperplane defines a split of $\Delta(d,n)$.
  In this case the intersection of $\Delta(d,n)$ with that split hyperplane equals
  \[
  \Delta(\rank(F),F)\times\Delta(d-\rank(F),[n]-F) \enspace ,
  \]
  and, in particular, the face $P_M(F)=P(M|F)\times P(M/F)$ is the intersection of $P(M)$ with the split hyperplane.
\end{lemma}
\begin{proof} 
  Pick an element $e\in[n]$ in the complement of $F$ which is not a coloop.
  This yields $\rank([n]-e)=d$, whence the submodularity of the rank function implies
  \begin{align*}
    \#F-\rank(F) \ &\leq \ \#F-\rank(F) + \#\bigl([n]-(F+ e)\bigr)-\rank\bigl([n]-(F+ e)\bigr)\\ 
      &\leq \ \#([n]-e) - \rank([n]-e)\\
      &= \ n-1 - d \enspace .
  \end{align*}
  With our assumption $0 < \rank(F) < \#F$ we obtain
  \[
  d-\#F \ < \ d-\rank(F) \ \leq \ n - \#F -1 \enspace ,
  \]
  which is precisely the condition in Proposition~\ref{prop:splits_of_the_hypersimplex} for $S=F$ and $\mu=d-\rank(F)$.
  This means that the $(F,d-\rank(F))$-hyperplane defines a split of $\Delta(d,n)$.
  The intersection with $\Delta(d,n)$ can be read off from \eqref{eq:product}.
\end{proof}
The value $d-\rank(F)$ is determined by the flat $F$, whence we will shorten the notation of $(F,d-\rank(F))$-hyperplane to \emph{$F$-hyperplane}.
Throughout the rest of this paper we will assume that $n\geq 2$, i.e., $M$ has at least two elements.
If $M$ is additionally connected, this forces that $M$ does not have any loops or coloops.
The relevance of the previous lemma for the investigation of matroid polytopes stems from the following observation.
\begin{proposition}\label{prop:mp-facets}
  Suppose that $M$ is connected.
  Each facet of $P(M)$ is defined by the $F$-hyperplane for some flat $F$ with $0< \rank(F) < \#F$, or it is induced by one of the hypersimplex facets.
  In particular, the facets of $P(M)$ are either induced by hypersimplex splits or hypersimplex facets.
\end{proposition}
\begin{proof}
  Consider an arbitrary facet $\Phi$ of the polytope $P(M)$.
  From \eqref{eq:matroid-exterior} we know that $\Phi$ is either induced by an inequality of the form $\sum_{i\in F}x_i\leq\rank(F)$ for some flat $F$ of $M$, or $\Phi$ corresponds to one of the non-negativity constraints.
  The latter yield hypersimplex facets, and the same also holds for the singleton flats.
  We are left with the case where $F$ has at least two elements.

  The connectivity implies that $M$ has no coloops, as we assumed that $M$ has at least two elements.
  Suppose that $\rank(F)=\# F$.
  Then the restriction $M|F$ to the flat consists of coloops and thus is disconnected.
  Since $M$ is connected, this implies that the hyperplane $\sum_{i\in F}x_i=\rank(F)$ cuts out a face of codimension higher than one.
  A similar argument works if $\rank(F)=0$ as in this case the contraction $M/F$ is disconnected.
  We conclude that $0 < \rank(F) < \#F$.
  Now the claim follows from Lemma~\ref{lem:flat-split}.
\end{proof}
We call a flacet $F$ a \emph{split flacet} if the \emph{$F$-hyperplane} is a split of $\Delta(d,n)$.
Notice that Lemma~\ref{lem:flat-split} explains this notion in matroid terms.
\begin{example}\label{example:snowflake-matroid}
  Let $\snow$ be the matroid on $n=6$ elements and rank $d=2$, with the three non-bases $12$, $34$ and $56$;
  i.e., $\snow$ has exactly twelve bases.
  We call this matroid the \emph{snowflake matroid} for its relationship with the snowflake tree discussed in Example~\ref{example:snowflake-tree} below.
  The pairs $12$, $34$ and $56$ form flats of rank one.
  The matroid polytope $P(\snow)$ has nine facets: the six non-negativity constraints $x_i\geq 0$, together with $x_1+x_2\leq 1$, $x_3+x_4\leq 1$ and $x_5+x_6\leq 1$.
  These are split flacets, written as in (\ref{eq:inhomogeneous}).
\end{example}

Two splits of a polytope $P$ are \emph{compatible} if their split hyperplanes do not meet in a relatively interior point of $P$.

\begin{definition}
  The $(d,n)$-matroid $M$ is a \emph{split matroid} if its split flacets form a compatible system of splits of the affine hull of $P(M)$ intersected with the unit cube $[0,1]^n$.
\end{definition}
The matroid polytopes of the $(d,n)$-split matroids are exactly those whose faces of codimension at least two are contained in the boundary of the $(d,n)$-hypersimplex.
The notion of a split matroid is a bit subtle in the disconnected case, which we will look into next.
See also Proposition~\ref{prop:disconnected} (which characterizes the connected components of a split matroid) and Example~\ref{example:compatible_splits} below.
\begin{lemma}\label{lem:disconnected}
  Let $M$ be a split matroid which is disconnected.
  Then each connected component of $M$ is a split matroid, too.
\end{lemma}
\begin{proof}
  Let $C$ be some connected component of the $(d,n)$-matroid $M$.
  Assume that $M|C$ has $n'=\#C$ elements and rank $d'$.
  Let $F$ and $G$ be two distinct split flacets of the connected matroid $M|C$.
  Notice that this can only happen if $M|C$ is not uniform.
  Now $F$ is a flat of $M$, and Lemma~\ref{lem:flat-split} gives us the $F$-hyperplane $H_F$ which yields a split of $\Delta(d,n)$ and a valid inequality of $P(M)$.
  Notice that we may assume that $[n]\setminus C$ contains an element which is not a coloop.
  We have
  \begin{equation}\label{eq:disconnected}
    \begin{aligned}
      H_F\cap\Delta(d,n) \ &= \ \Delta(\rank(F),F)\times\Delta(d-\rank(F),[n]-F) \\
      &= \ \Delta(\rank(F),F)\times\Delta(d'-\rank(F),C-F)\times\Delta(d-d',[n]-C)\enspace .
    \end{aligned}
  \end{equation}
  That intersection contains interior points of $\Delta(d,n)$, which is why this defines a facet of $P(M)$.
  By construction this defines a split flacet of $M$.
  The same applies to $G$, yielding another split hyperplane $H_G$, which also yields a split flacet of $M$.
  Since $M$ is a split matroid these two split flacets of $M$ are compatible.
  The explicit description in (\ref{eq:disconnected}) shows that the split flacets $F$ and $G$ of $M|C$ are compatible, too.
  We conclude that $M|C$ is a split matroid.
\end{proof}
We conclude that it suffices to analyze those split matroids which are connected.
The following characterization of split matroids does not require any reference to polyhedral geometry.
\begin{theorem}\label{thm:split-matroid}
  Let $M$ be a connected matroid.
  The matroid $M$ is a split matroid if and only if for each split flacet $F$ the restriction $M|F$ and the contraction $M/F$ both are uniform.
\end{theorem}
\begin{proof}
  Assume that $M$ is a split matroid and $F$ is a split flacet.
  Let $r$ be the rank of $F$. As $F$ does not correspond to a hypersimplex facet we know that $r<d$. Hence $F$ is not the entire ground set $[n]$.
  In particular, all conditions for Lemma~\ref{lem:flat-split} are satisfied.
   Moreover, the intersection of any two facets of the matroid polytope $P(M)$ is contained in the boundary of the hypersimplex $\Delta(d,n)$.
   This implies that the intersection of the split hyperplane of $F$ with $P(M)$ coincides with the intersection of that hyperplane with $\Delta(d,n)$.
  By Lemma~\ref{lem:flat-split} we have that $M|F$ is the uniform matroid of rank $r$ on the set $F$, and $M/F$ is the uniform matroid of rank $d-r$ on the set $[n]-F$.

  To prove the converse, let $F$ and $G$ be two distinct split flacets of $M$ with uniform restrictions and contractions.
  We need to show that the hypersimplex splits corresponding to $F$ and $G$ are compatible.
  By Proposition~\ref{prop:matroid-polytope}(\ref{it:matroid-polytope:product}) and Lemma~\ref{lem:flat-split} we have
  \begin{equation}\label{eq:FG}
     P_M(F) \ = \ P(M|F)\times P(M/F) \ = \ \Delta(\rank(F),F)\times\Delta(d-\rank(F),[n]-F) \enspace .
  \end{equation}
  This implies that $P_M(F)$ is exactly the intersection of the $F$-hyperplane with $\Delta(d,n)$.
  In particular, since the $G$-hyperplane is a valid inequality for $P_M(F)$, the $F$- and $G$-hyperplanes do not share any points in the relative interior of $\Delta(d,n)$.
  This means that the corresponding hypersimplex splits are compatible.
\end{proof}

\begin{remark}\label{rem:partition}
  Equation~\eqref{eq:FG} says that the face $P_M(F)$ corresponding to a flacet $F$ of split a matroid is the matroid polytope of a partition matroid, i.e., a direct sum of uniform matroids.
\end{remark}
  
A flat is called \emph{cyclic} if it is a union of circuits.
This notion gives rise to yet another cryptomorphic way of defining matroids; see \cite[Theorem 3.2]{BoninMier:2008}.
A matroid whose cyclic flats form a chain with respect to inclusion is called \emph{nested}.
Such matroids will play a role in Section~\ref{sec:rays} below.
\begin{proposition}\label{prop:cyclic_flat}
  Each flacet $F$ of $M$ with at least two elements is a cyclic flat.
  This property holds even if $M$ is not connected.
\end{proposition}
\begin{proof}
  Let $F$ be a flacet of $M$.
  The restriction $M|F$ is connected, even if $M$ itself is not connected, see also Remark~\ref{rem:flacets}.
  Thus for each $e\in F$ there exists a circuit $e\in C\subseteq F$ in $M|F$ that connects $e$ with an other element of $F$.
  This circuit of $M|F$ is a minimal dependent set in $M$.
  Hence $F$ a cyclic flat.
\end{proof}

The compatibility relation among the hypersimplex splits was completely described in \cite[Proposition 5.4]{HerrmannJoswig:2008}.
The following is a direct consequence.
Notice that this characterization of split compatibility is a tightening of the submodularity property of the rank function.

\begin{proposition}\label{prop:compatible_splits}
  Assume that $M$ is connected.
  Let $F$ and $G$ be two distinct split flacets.
  The splits obtained from the $F$- and the $G$-hyperplane are compatible if and only if
  \[
  \size(F\cap G)+d \ \leq \ \rank(F)+\rank(G) \enspace .	
  \]
  For instance, this condition is satisfied if $F\cap G$ is an independent set and $F + G$ contains a basis.
\end{proposition}
\begin{proof}
  The $F$- and the $G$-hyperplane both define splits.
  \cite[Proposition 5.4]{HerrmannJoswig:2008} states that two splits are compatible if and only if exactly one of the following four inequalities hold.
  \begin{align*}
	\size(F\cap G) \ &\leq \ \rank(F)+\rank(G)-d\\
	\size(F-G) \ &\leq \ \rank(F)-\rank(G)\\
       	\size(G-F) \ &\leq \ \rank(G)-\rank(F)\\
       	\size([n]-F-G) \ &\leq \ d-\rank(F)-\rank(G)
  \end{align*}
  We will show that the last three conditions never hold for a connected matroid.

  We denote by $H\subseteq F\cap G$ the inclusion maximal cyclic flat that is contained in $F\cap G$.
  Then $c:=\size(F\cap G)-H$ is the number of coloops in $M|(F\cap G)$.
  By Proposition~\ref{prop:cyclic_flat} the flacet $F$ is a cyclic flat, too.
  Now \cite[Theorem 3.2]{BoninMier:2008} implies that
  \[
  \size(F-G) \ = \ \size(F-H)-c \ > \ \rank(F)-\rank(H)-c \ = \ \rank(F)-\rank(G\cap F) \ \geq \ \rank(F)-\rank(G) \enspace .
  \] 
  Similarly we get $\size(G-F) > \rank(G)-\rank(F)$.
  The submodularity of the rank function yields
  \[
  \begin{aligned}
    \size([n]-(F+G))  + \rank(F) + \rank(G) -d &\geq \rank([n]-(F+G)) + \rank(F+G) + \rank(F\cap G) - d\\
                                                   &\geq \rank([n])-d +\rank(F\cap G)\\ 
                                                   &\geq 0 \enspace .
  \end{aligned}
  \]
  In the above equality holds if and only if the matroid is the direct sum $F\oplus G \oplus ([n]-(F+G))$ and the set $[n]-(F+G)$ consists of coloops.
 
  If $F\cap G$ is independent and $F+G$ has full rank $d$ we have
  \begin{equation}\label{eq:compatible}
    \size(F\cap G) \ = \ \rank(F\cap G)+\underbrace{\rank(F + G)-d}_{=\ 0} \ \leq \ \rank(F)+\rank(G)-d \enspace .
  \end{equation}
\end{proof}

\begin{proposition}\label{prop:disconnected}
   A matroid $M$ is a split matroid if and only if at most one connected component is a non-uniform split matroid and all other connected components are uniform.
\end{proposition}
\begin{proof}
   We only need to discuss the case that $M$ is disconnected.
   First assume that $M$ is a direct sum of uniform matroids and at most one non-uniform split matroid $M|C$.
   Let $F$ and $G$ be a split flacets of $M$.
   By assumption the $F$-hyperplane does not separate the matroid polytope of any of the uniform matroids.
   Hence $F$ is a flacet of $M|C$.
   Similarly is $G$ a flacet of the split matroid $M|C$.
   In particular, the intersection of the $F$-hyperplane with the $G$-hyperplane restricted to $P(M|C)$ contains no interior point of $P(M|C)$.
   This implies that the intersection of the $F$-hyperplane with the $G$-hyperplane contains no interior point of $P(M)=P(M|C)\times P(M/C)$.

   Now assume that $M$ is a disconnected split $(d,n)$-matroid.
   From Lemma~\ref{lem:disconnected} we know that each connected component is a split matroid.
   Let $C_1$, $C_2$ be two connected components of $M$, and let $F,G$ be a split flacets of $C_1$ and $C_2$, respectively.
   These split flacets exist if and only if neither $M|C_1$ nor $M|C_2$ is uniform.
   Let $x_F\in P(M|C_1)$ be a point on the relative interior of the facet defined by $\sum_{i\in F} x_i = \rank(F)$.
   Similarly, let $x_G\in P(M|C_2)$ be a point on the relative interior of facet defined by $G$.
   Finally, let $x_H$ be a point in the relative interior of $P(M/(C_1+C_2))$.

   We have seen in Lemma~\ref{lem:disconnected} that the $F$-hyperplane is a facet of $P(M)$.
   Hence is $F$ a flacet of $M$, and $G$ is similar.
   By construction the point $(x_F,x_G,x_H)\in P(M|C_1)\times P(M|C_2)\times P(M/(C_1+C_2))$ lies in the interior of $P(M)$ as well as on the $F$- and $G$-hyperplanes.
   We conclude that the flacets $F$ and $G$ are incompatible.
   Since this cannot happen in a split matroid, we may conclude that either $M|C_1$ or $M|C_2$ are uniform.
\end{proof}

\begin{example}\label{example:direct_sum}
  For instance, the direct sum of the $(2,4)$-matroid with five bases, which is a split matroid, with an isomorphic copy is not a split matroid.
\end{example}

\begin{example} \label{example:compatible_splits}
  The $12$-, $34$- and the $56$-hyperplanes, corresponding to the split flacets of the snowflake matroid $\snow$ from Example~\ref{example:snowflake-matroid} are pairwise compatible.
  For instance, we have $\size(\{1,2\}\cap\{3,4\})=0\leq 1+1-2$.
  This shows that the snowflake matroid is a split matroid; see also Figure~\ref{subfig:snow} below.
  Note that the direct sum of the snowflake matroid with a coloop $U_{1,1}$ is a split matroid, too.
  In particular, the $12$- and $34$-hyperplanes do not intersect in the interior of $\Delta(2,6)\times\Delta(1,1)$.
  However, they do intersect in the interior of $\Delta(3,7)$, as $\size(\{1,2\}\cap\{3,4\})=0 > 1+1-3$ shows.
\end{example}

\begin{example}\label{example:non_split_matroid}
  For a different kind of example consider the $(3,6)$-matroid with the eight non-bases $134$, $234$, $345$, $346$, $156$, $256$, $356$ and $456$.
  This matroid has exactly the two flacets $34$ and $56$. The $34$- and the $56$-hyperplanes are not compatible.
  Hence this is not a split matroid.
\end{example}

A rank-$d$ matroid whose circuits have either $d$ or $d+1$ elements is a \emph{paving} matroid.
It is conjectured that asymptotically almost all matroids are paving; see \cite[Conjecture 15.5.10]{Oxley:2011} and \cite[Conjecture 1.6]{MayhewNewmanWelshWhittle:2011}.
A paving matroid whose dual is also paving is called \emph{sparse paving}.
It is known that a matroid is paving if and only if there is no minor isomorphic to the direct sum of the uniform matroid $U_{2,2}$ and $U_{0,1}$; see \cite[page 126]{Oxley:2011}.
The following is a geometric characterization of the paving matroids.

\begin{theorem}\label{thm:paving_matroids}
  Suppose that the $(d,n)$-matroid $M$ is connected.
  Then $M$ is paving if and only if it is a split matroid such that each split flacet has rank $d-1$.
\end{theorem}
\begin{proof}
  Let $M$ be paving, and let $F$ be a split flacet.
  Then $F$ is a corank-1 flat of $M$, i.e., $F$ is a proper flat of maximal rank $d-1$.
  Since there are no circuits with fewer than $d$ elements, the restriction $M|F$ is a uniform matroid of rank $d-1$.
  The contraction $M/F$ is a loop-free matroid of rank~$1$, and thus uniform.
  By Theorem~\ref{thm:split-matroid} we find that $M$ is a split matroid, and each split flacet of $M$ has rank $d-1$.

  Conversely, let $M$ be a matroid such that the split flacets correspond to a compatible system of splits of $\Delta(d,n)$ such that, moreover, each split flacet is of rank $d-1$.
  Let $F$ be such a split flacet.
  Then, by Lemma~\ref{lem:flat-split} we have $P_M(F)=\Delta(d-1,F)\times\Delta(1,[n]-F)$.
  It follows that the restriction $M|F$ does not have a circuit with fewer than $d$ elements.

   Now consider a set $C$ of size $d-1$ or less which is contained in no split flacet, and let $D\subseteq[n]-C$ be some set of size $d-\#C$ in the complement of $C$.
  Let $\bar x=e_{C+D}$.
  Then, for any flacet $F$, we have
  \[
  \sum_{i\in F} \bar x_i \ = \ \sum_{i\in F\cap C} \bar x_i + \sum_{i\in F\cap D} \bar x_i \ < \ \#C + d - \#C \ = \ d \enspace .
  \]
  as $C$ is not contained in $F$.
  This shows that $\bar x$ satisfies the flacet inequality $\sum_{i\in F} x_i\leq d-1$.
  Further, the inequalities imposed by the hypersimplex facets also hold, and so $\bar x$ is contained in $P(M)$.
  Since $\bar x=e_{C+D}$ is a vertex of $\Delta(d,n)$ it follows that it must also be a vertex of the subpolytope $P(M)$.
  Therefore, $C+D$ is a basis of $M$, whence $C$ is an independent set.
  We conclude that $M$ does not have any circuit with fewer than $d$ elements.
  Any circuit of a rank-$d$ matroid with more than $d$ elements has exactly $d+1$ elements.
  This is why $M$ is a paving matroid.
\end{proof}


\begin{remark}
  Each split flacet of a paving matroid $M$ corresponds to a partition matroid, and the split flacets are precisely the corank-1 flats of $M$ that contain a circuit; see also Remark~\ref{rem:partition}.
  In this way, the split flacets of a paving matroid implicitly occur in the matroid literature, e.g., in the proof of \cite[Prop. 2.1.24]{Oxley:2011}.
\end{remark}

We want to look into a construction which yields very many split matroids.
Let $\sigma$ be some $d$-element subset of $[n]$.
That is, $\sigma$ is a basis of the uniform matroid $U_{d,n}$, and $e_\sigma=\sum_{i\in\sigma}e_i$ is a vertex of $\Delta(d,n)$.
It neighbors in the Johnson graph $J(d,n)$ lie on the $(\sigma,d-1)$-hyperplane in $\Delta(d,n)$.
More precisely, from \eqref{eq:product} we can see that the convex hull of the neighbors of $e_\sigma$ equals
\[
\Delta(d-1,\sigma)\times\Delta(1,[n]-\sigma) \enspace ,
\]
which is the product of a $(d-1)$-simplex and an $(n-d-1)$-simplex.
The resulting split is called the \emph{vertex split} with respect to $\sigma$ or $e_\sigma$.
Two vertex splits are compatible if and only if the two vertices do not span an edge.
In this way the compatible systems of vertex splits of $\Delta(d,n)$ bijectively correspond to the stable sets in the Johnson graph $J(d,n)$.
The following observation is similar to \cite[Lemma 8]{BansalPendavinghVanDerPol:2012}.
\begin{corollary}\label{cor:sparse-paving}
  Again let $M$ be a $(d,n)$-matroid which is connected.
  Then $M$ is sparse paving if and only if the conclusion of Theorem~\ref{thm:paving_matroids} holds and additionally the splits are vertex splits.
\end{corollary}
\begin{proof}
   For each rank $d-1$ split flacet $F$ of  the split matroid $M$ we have $M/F = U_{1,[n]-F}$ and $M|F = U_{d-1,F}$.
   The dual of $M$ is a matroid of rank $n-d$ on $n$ elements. The matroid polytope $P(M^*)$ is the image of $P(M)$ under coordinate-wise transformation $x_i\mapsto 1-x_i$.
   It follows that the split flacets of $M^*$ are the complements of the split flacets of $M$.
   Thus, for the split flacet $[n]-F$ in $M^*$, we obtain
  \begin{align*}
    M^*|([n]-F) \ &= \ (M/F)^* \ = \ U^*_{1,[n]-F} \ = \ U_{ n-\# F-1,[n]-F} \quad \text{and}\\
    M^*/([n]-F) \ &= \ (M|F)^* \ = \ U^*_{d-1,F} \ = \ U_{\# F-d+1,F} \enspace .
  \end{align*}
  This implies that $M^*$ is paving if and only if each split flacet $F$ has cardinality $d$.
\end{proof}

The following two examples illustrate the differences between paving and split matroids.
The class of split matroids is strictly larger.
In contrast to the class of paving matroids the class of split matroids is closed under dualization.
\begin{example}
  The $(\{1,2,3,4\},2)$-hyperplane yields a split of the hypersimplex $\Delta(4,8)$.
  The two maximal cells correspond to split matroids which are not paving nor are their duals.
\end{example}

Yet there are still plenty of matroids which are not split.
\begin{example}\label{example:non_split_nested}
  Up to symmetry there are 15 connected matroids of rank three on six elements.
  Among these there are exactly four which are non split.
  One such example is the nested matroid given by the columns of the matrix
  \[
  \begin{pmatrix}
    1 & 1 & 0 & 1 & 0 & \lambda \\
    0 & 0 & 1 & 1 & 0 & 1 \\
    0 & 0 & 0 & 0 & 1 & 1
  \end{pmatrix} \enspace .
  \]
  where $\lambda\neq0,1$.
  This matroid is realizable over any field with more than two elements.
\end{example}

Knuth gave the following construction for stable sets in Johnson graphs \cite{Knuth:1974}.
Due to Corollary~\ref{cor:sparse-paving} this is the same as a compatible set of vertex splits, which arise from the split flacets of a sparse paving matroid.
\begin{example}\label{example:knuth}
  The function
  \[
  (x_1,\ldots,x_n)\mapsto\sum_{i=1}^n i\cdot x_i \mod n
  \]
  defines a proper coloring of the nodes of $J(d,n)$ with $n$ colors.
  Each color class forms a stable set, and there must be at least one stable set of size at least $\frac{1}{n}\tbinom{n}{d}$.
\end{example}
For special choices of $d$ and $n$ larger stable sets in $J(d,n)$ are known.
\begin{example}
  Identifying a natural number between $0$ and $2^k-1$ with its binary representation yields a $0/1$-vector of length $k$.
  All quadruples of such vectors that sum up to $0$ modulo $2$ form a stable set $S$ in $J(4,2^k)$ of size $n(n-1)(n-2)/24$, where $n=2^k$.
  Fixing one vector and restricting to those quadruples in $S$ which contain that vector gives a stable set of size $(n-1)(n-2)/6$ in $J(3,2^k-1)$.
  The latter construction also occurs in \cite[Theorem 3.1]{Dukes:2004} and \cite[Theorem 3.6]{HerrmannJensenJoswigSturmfels:2009}.
\end{example}

In a way, the sparse paving matroids are those split matroids which are the easiest to get at.
We sum up our discussion in the following characterization.
\begin{theorem}
  The following sets are in bijection with one another:
  \begin{enumerate}
  \item[(i)] The split flacets of sparse paving connected matroids of rank $d$ on $n$ elements,
  \item[(ii)] the cyclic flats of sparse paving connected matroids of rank $d$ on $n$ elements,
  \item[(iii)] the sets of compatible vertex splits of $\Delta(d,n)$,
  \item[(iv)] the stable sets of the graph $J(d,n)$,
  \item[(v)] the sets of binary vectors of length $n$ with constant weight $d$ and Hamming distance at least $4$.
  \end{enumerate}
\end{theorem}
\begin{proof}
  Each split flacet of a sparse paving matroid $M$ is a cyclic flat by Proposition~\ref{prop:cyclic_flat}. 
  The proof of Theorem~\ref{thm:paving_matroids} shows that cyclic flats of rank $d-1$ are split flacets of $M$.
  Further, the cyclic flats of a connected paving matroid are those of rank $d-1$, the empty set and the entire ground set $[n]$.
  This establishes that (i) and (ii) are equivalent.

  Corollary~\ref{cor:sparse-paving} is exactly the equivalence of (i) and (iii).

  By Proposition~\ref{prop:compatible_splits} two vertex splits of $\Delta(d,n)$ are compatible if and only if the two vertices do not span an edge.
  The compatible systems of vertex splits of $\Delta(d, n)$ bijectively correspond to the stable sets in the vertex-edge graph of $\Delta(d,n)$, which is the Johnson graph $J(d, n)$.
  This means that (iii) is equivalent to (iv).

  The vertices of the hypersimplex $\Delta(d, n)$ are all binary vectors of length $n$ with constant weight $d$.
  The Hamming distances of two such vectors $v$ and $w$ is the number of coordinates where $v_i\neq w_i$.
  This number is twice the distance of the vertices in the Johnson graph $J(d,n)$.
  Note that odd numbers do not occur as Hamming distances.
  Hamming distance at least $4$ means that the vertices are not adjacent in $J(d,n)$.
  This yields the equivalence of (iv) and (v).
\end{proof}

A table with lower bounds on the maximal size of such a set for $n \leq 28$ is given in \cite[Table I-A]{Brouweretal:1990}.
Notice that this data also gives lower bounds on the total number of $(d,n)$-matroids; see, e.g., \cite{BansalPendavinghVanDerPol:2012}.

\section{Matroid subdivisions and tropical linear spaces}\label{sec:Dr}
\noindent
In this section we want to exploit the structural information that we gathered about split matroids to derive new results about tropical linear spaces, the tropical Grassmannians and the related Dressians \cite{SpeyerSturmfels:2004,HerrmannJensenJoswigSturmfels:2009}.
We begin with some basics on general polyhedral subdivisions; see \cite{LoeraRambauSantos:2010} for further details.

Let $P$ be some polytope.
A polytopal subdivision of $P$ is \emph{regular} if it is induced by a lifting function on the vertices of $P$.
Examples are given by the Delaunay subdivisions where the lifting function is the Euclidean norm squared.
The lifting functions on $P$ which induce the same polytopal subdivision, $\Sigma$, form a relatively open polyhedral cone, the \emph{secondary cone} of $\Sigma$.
The \emph{secondary fan} of $P$ comprises all secondary cones.
The inclusion relation on the closures of the secondary cones of $P$ imposes a partial ordering, and this is dual to the set of regular polytopal subdivisions of $P$ partially ordered by refinement.
The secondary fan has a non-trivial lineality space which accounts for the various choices of affine bases.
Usually we will ignore these linealities.
In particular, whenever we talk about dimensions we refer to the dimension of a secondary fan modulo its linealities.

A \emph{tropical Pl\"ucker vector} $\pi\in\RR^{n \choose d}$ is a lifting function on the vertices of the hypersimplex $\Delta(d,n)$ such that the regular subdivision induced by $\pi$ is a \emph{matroid subdivision}, i.e., each of its cells is a matroid polytope. The cells of the dual of a matroid subdivision that correspond to loop-free matroid polytopes form a subcomplex.
This subcomplex of that matroid subdivision is the \emph{tropical linear space} defined by $\pi$.
The \emph{Dressian} $\Dr(d,n)$ is the subfan of the secondary fan of the hypersimplex $\Delta(d,n)$ comprising the tropical Pl\"ucker vectors.
According to Remark \ref{rem:splits-of-matroid-polytopes} each split of a hypersimplex is a regular matroid subdivision and hence it defines a ray of the corresponding Dressian.

Let $M$ be a $(d,n)$-matroid.
The matroid polytope $P(M)$ is a subpolytope of $\Delta(d,n)$.
Restricting the tropical Pl\"ucker vectors to vertices of $P(M)$ and looking at regular subdivisions of $P(M)$ into matroid polytopes gives rise to the \emph{Dressian} $\Dr(M)$ of the matroid $M$; see \cite[Section~6]{HerrmannJensenJoswigSturmfels:2009}.
The rank of any subset $S$ of $[n]$ coincides with the rank of the flat spanned by $S$.
Restricting the rank function of $M$ to all subsets of $[n]$ of a fixed cardinality $k$ yields the $k$-\emph{rank vector} of $M$.
The \emph{dual-rank function} of $M$ is the rank function of $M^*$, the dual matroid of $M$, and the \emph{corank function} is the difference between $d$ and the rank function.
The $k$-\emph{corank vector} of $M$ is the map
\[
\rho_k(M) \,:\, \binom{[n]}{k} \to \NN \,, \ S\mapsto d-\rank_M(S) \enspace .
\]
The regular subdivision of $\Delta(k,n)$ with lifting function $\rho_k(M)$ is the $k$-\emph{corank subdivision} induced by the matroid $M$.
Usually we will omit the size $k$ in those definitions if $k$ equals $d$.
The following known result says that the $k$-corank subdivision is a matroid subdivision.
\begin{lemma}\label{lem:corank-cell}
  The $k$-corank vector $\rho_k(M)$ of the $(d,n)$-matroid $M$ is a $(k,n)$-tropical Pl\"ucker vector.
  Moreover, the matroid polytope $P(M)$ occurs as a cell in the $k$-corank subdivision induced by $M$.
  That cell is maximal if and only if $M$ is connected.
\end{lemma}
\begin{proof}
  Speyer showed that  $\rho_k(M)$ is a tropical Pl\"ucker vector such that the matroid polytope $P(M)$ occurs as a cell \cite[Proposition 4.5.5]{Speyer:2005}.
  The dimension of that cell can be read off from Proposition~\ref{prop:matroid-polytope}.
\end{proof}

\begin{example}
  With $d=2$ and $n=4$ let $M$ be the matroid with the five bases $12$, $13$, $14$, $23$ and $24$.
  We pick $k=d=2$.
  The rank of the unique non-basis $34$ equals $1$, whence $\rho_2(M)=(0,0,0,0,0,1)$.
  The matroid subdivision induced by $\rho_2(M)$ splits the hypersimplex $\Delta(2,4)$ into two Egyptian pyramids.
  Every subset of $\{1,2,3,4\}$ with cardinality $k=3$ contains a basis, and thus $\rho_3(M)=(0,0,0,0)$.
  There are no loops in $M$, whence for $k=1$ the corank vector $\rho_1(M)$ equals $(1,1,1,1)$.
  Here and below the ordering of the $k$-subsets of $[n]$ in the corank vectors is lexicographic.
\end{example}

\begin{example}\label{example:snowflake-tree}
  The corank  subdivision of the matroid $\snow$ in Example~\ref{example:snowflake-matroid} is a matroid subdivision of $\Delta(2,6)$ whose tropical linear space is the snowflake tree.
  Hence the name snowflake matroid for $\snow$.
  See Figure~\ref{subfig:snow} for a visualization.
\end{example}

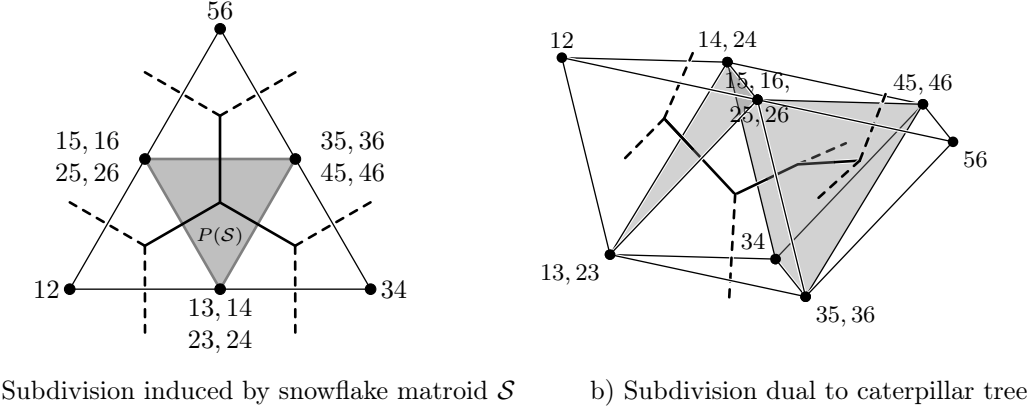
\begin{figure}
   \begin{subfigure}[c]{0.49\textwidth}
   \begin{tikzpicture}[scale =2, font=\small] 
   \draw[white] (-0.7,-0.5) -- (2.5,2);
   
  \coordinate (a0) at (0, 0);
  \coordinate (a2) at (2,0);
  \coordinate (c) at (1,1.732);
   \coordinate (a1) at ($(a0)!0.5!(a2)$);
   \coordinate (b0) at ($(a0)!0.5!(c)$);
   \coordinate (b1) at ($(a2)!0.5!(c)$);
   \coordinate (t0) at ($0.333*(a0)+0.333*(c)+0.333*(a2)$);
   \coordinate (t1) at ($0.333*(a0)+0.333*(b0)+0.333*(a1)$);
   \coordinate (t2) at ($0.333*(a1)+0.333*(b1)+0.333*(a2)$);
   \coordinate (t3) at ($0.333*(b0)+0.333*(b1)+0.333*(c)$);

   \coordinate (o1) at ($0.333*(a0)+(b0)-0.333*(b1)$);
   \coordinate (o2) at ($0.333*(a0)+(a1)-0.333*(b1)$);
   \coordinate (o3) at ($0.333*(a2)+(a1)-0.333*(b0)$);
   \coordinate (o4) at ($0.333*(a2)+(b1)-0.333*(b0)$);
   \coordinate (o5) at ($0.333*(c)+(b0)-0.333*(a1)$);
   \coordinate (o6) at ($0.333*(c)+(b1)-0.333*(a1)$);

 \tikzstyle{hypersimplex} = [fill=none, fill opacity=0.85, draw=black, line width=0.5 pt, line cap=round, line join=round]
 \tikzstyle{matroid} = [fill=gray, fill opacity=0.5, draw=gray, line width=1 pt, line cap=round, line join=round]
 \tikzstyle{tree} = [fill=none, draw=black, line width=1 pt, line cap=round, line join=round]
 \tikzstyle{vertex} = [fill=black]

   \draw[hypersimplex] (a0) -- (a2) -- (c) -- cycle;
   \draw[matroid] (a1) -- (b0) -- (b1) -- cycle;
   \draw[tree] (t1) -- (t0) -- (t2);
   \draw[tree] (t0) -- (t3);

   \draw[tree, dashed] (t1) -- (o1);
   \draw[tree, dashed] (t1) -- (o2);
   \draw[tree, dashed] (t2) -- (o3);
   \draw[tree, dashed] (t2) -- (o4);
   \draw[tree, dashed] (t3) -- (o5);
   \draw[tree, dashed] (t3) -- (o6);

   \draw[vertex] (a0) circle (1pt);
 \draw[vertex] (a1) circle (1pt);
 \draw[vertex] (a2) circle (1pt);
 \draw[vertex] (b0) circle (1pt);
 \draw[vertex] (b1) circle (1pt);
 \draw[vertex] (c) circle (1pt);

   \node[below] at ($(t0)-(0,0.1)$) {\tiny $P(\snow)$};
   \node[left] at (a0) {$12$};
   \node[right] at (a2) {$34$};
   \node[above] at (c) {$56$};

   \node[below, align=center] at (a1) {$13$,\,$14$\\ $23$,\,$24$};
   \node[right, align=center] at ($(b1)+(0.1,0)$) {$35$,\,$36$\\ $45$,\,$46$};
   \node[left, align=center] at ($(b0)-(0.1,0)$) {$15$,\,$16$\\ $25$,\,$26$};

\end{tikzpicture}
   \subcaption{Subdivision induced by snowflake matroid $\snow$}
   \label{subfig:snow}
   \end{subfigure}
   \begin{subfigure}[c]{0.49\textwidth}

\begin{tikzpicture}[x  = {(-0.71cm,0.67cm)},
                    y  = {(-0.39cm,-0.64cm)},
                    z  = {(0.59cm,0.39cm)},
                    scale = 2,
                    ]
                    
   \draw[white] (1.05,2.19) -- (-0.45,-3.28);

   \tikzstyle{pre} = [preaction={draw=white, line cap=round, line join=round, line width=1.3 pt}]
   \tikzstyle{hypersimplex} = [pre, fill=none, fill opacity=0.85, draw=black, line width=0.5 pt, line cap=round, line join=round]
   \tikzstyle{cell} = [pre, fill=gray, fill opacity=0.35, draw=black, line width=0.5 pt, line cap=round, line join=round]
   \tikzstyle{tree} = [preaction={draw=white, line cap=butt, line join=round, line width=1.7 pt}, fill=none, draw=black, line width=1 pt, line cap=round, line join=round]
   \tikzstyle{tree2} = [draw=black, line width=1 pt, line cap=round, line join=round]
   \tikzstyle{tree3} = [tree, dashed, style={shorten >=-0.4cm}]

   \tikzstyle{tree3b} = [preaction={draw=gray!35!white, line cap=butt, line join=round, line width=1.7 pt}, fill=none, draw=black, line width=1 pt, line cap=round, line join=round, dashed ]

   \tikzstyle{vertex} = [fill=black]

  \coordinate (v0__3) at (1, -1, 0);
  \coordinate (v1__3) at (1, 0, 1);
  \coordinate (v2__3) at (0, 0, 0);
  \coordinate (v3__3) at (0, 1, 1);
  \coordinate (v4__3) at (0, -1, 1);
  \coordinate (t1) at ($0.2*(v0__3)+0.2*(v1__3)+0.2*(v2__3)+0.2*(v3__3)+0.2*(v4__3)$);
  \coordinate (v0__1) at (1, 1, 0);
  \coordinate (v1__1) at (1, -1, 0);
  \coordinate (v2__1) at (1, 0, 1);
  \coordinate (v3__1) at (0, 0, 0);
  \coordinate (v4__1) at (0, 1, 1);
  \coordinate (t2) at ($0.2*(v0__1)+0.2*(v1__1)+0.2*(v2__1)+0.2*(v3__1)+0.2*(v4__1)$);
  \coordinate (v0__2) at (1, 0, 1);
  \coordinate (v1__2) at (0, 1, 1);
  \coordinate (v2__2) at (0, -1, 1);
  \coordinate (v3__2) at (0, 0, 2);
  \coordinate (t0) at ($0.25*(v0__2)+0.25*(v1__2)+0.25*(v2__2)+0.25*(v3__2)$);
  \coordinate (v0__4) at (1, 1, 0);
  \coordinate (v1__4) at (1, -1, 0);
  \coordinate (v2__4) at (1, 0, 1);
  \coordinate (v3__4) at (2, 0, 0);
  \coordinate (t3) at ($0.25*(v0__4)+0.25*(v1__4)+0.25*(v2__4)+0.25*(v3__4)$);

  \draw[tree3]  (t1) -- ($0.33*(v0__3)+0.33*(v3__3)+0.33*(v4__3)$);
  \draw[hypersimplex] (v4__3) -- (v0__3) -- (v1__3) -- (v4__3) -- cycle;
  \draw[hypersimplex] (v3__3) -- (v2__3) -- (v4__3) -- (v3__3) -- cycle;
  \draw[hypersimplex] (v4__3) -- (v2__3) -- (v0__3) -- (v4__3) -- cycle;
  \draw[tree]  ($(t2)!0.5!(t1)$) -- (t1) -- ($(t1)!0.6!(t0)$);
  \fill[black] (v4__3) circle (1 pt);
  \fill[black] (v0__3) circle (1 pt);
  \fill[black] (v2__3) circle (1 pt);
  \fill[black] (v3__3) circle (1 pt);
  \fill[black] (v1__3) circle (1 pt);

  \draw[hypersimplex] (v3__1) -- (v0__1) -- (v1__1) -- (v3__1) -- cycle;
  \draw[cell] (v3__1) -- (v1__1) -- (v2__1) -- (v4__1) -- (v3__1) -- cycle;
  \draw[hypersimplex] (v3__1) -- (v4__1) -- (v0__1) -- (v3__1) -- cycle;
  \draw[tree]  ($(t1)!0.5!(t2)$) -- (t2) -- ($(t2)!0.6!(t3)$);
  \draw[tree3]  (t2) -- ($0.33*(v0__1)+0.33*(v3__1)+0.33*(v4__1)$);
  \fill[black] (v3__1) circle (1 pt);
  \fill[black] (v1__1) circle (1 pt);
  \draw[hypersimplex] (v2__1) -- (v0__1) -- (v4__1) -- (v2__1) -- cycle;
  \fill[black] (v2__1) circle (1 pt);
  \fill[black] (v0__1) circle (1 pt);
  \fill[black] (v4__1) circle (1 pt);
  
  \draw[hypersimplex] (v0__2) -- (v3__2) -- (v2__2) -- (v0__2) -- cycle;
  \draw[hypersimplex] (v2__2) -- (v3__2) -- (v1__2) -- (v2__2) -- cycle;
  \draw[cell] (v1__2) -- (v0__2) -- (v2__2) -- (v1__2) -- cycle;
  \draw[tree]  ($(t1)!0.6!(t0)$) -- (t0);
  \draw[tree3b,  style={shorten >=-0.4cm}]  (t0) -- ($0.33*(v0__2)+0.33*(v1__2)+0.33*(v3__2)$);
  \fill[black] (v2__2) circle (1 pt);
  \draw[tree3]  (t0) -- ($0.33*(v0__2)+0.33*(v2__2)+0.33*(v3__2)$);
  \draw[tree3b, style={shorten >=-0.19cm}]  (t0) -- ($0.33*(v0__2)+0.33*(v2__2)+0.33*(v3__2)$);
  \draw[hypersimplex] (v1__2) -- (v3__2) -- (v0__2) -- (v1__2) -- cycle;
  \fill[black] (v1__2) circle (1 pt);
  \fill[black] (v3__2) circle (1 pt);
  \fill[black] (v0__2) circle (1 pt);

  \draw[cell] (v1__4) -- (v2__4) -- (v0__4) -- cycle;
  \draw[hypersimplex] (v0__4) -- (v3__4) -- (v1__4) -- (v0__4) -- cycle;
  \draw[tree]  ($(t2)!0.6!(t3)$) -- (t3);
  \draw[hypersimplex] (v3__4) -- (v2__4) -- (v1__4) -- (v3__4) -- cycle;
  \draw[tree3]  (t3) -- ($0.33*(v1__4)+0.33*(v2__4)+0.33*(v3__4)$);
  \draw[hypersimplex] (v3__4) -- (v2__4);
  \fill[black] (v1__4) circle (1 pt);
   \draw[tree3]  (t3) -- ($0.33*(v0__4)+0.33*(v2__4)+0.33*(v3__4)$);
  \draw[hypersimplex] (v0__4) -- (v2__4) -- (v3__4) -- (v0__4) -- cycle;
  \draw[tree2]  ($(t2)!0.6!(t3)$) -- (t3);
  \fill[black] (v0__4) circle (1 pt);
  \fill[black] (v2__4) circle (1 pt);
  \fill[black] (v3__4) circle (1 pt);

   \node[below right] at (v4__1) {\footnotesize $35$,\,$36$};
   \node[above left] at (v3__1) {\footnotesize $34$};
   \node[below right] at (v3__2) {\footnotesize $56$};
   \node[above] at (v2__2) {\footnotesize $45$,\,$46$};
   \node[above] at (v3__4) {\footnotesize $12$};
   \node[align=center] at (v2__4) {\footnotesize $15$,\,$16$,\\ \footnotesize \,$25$,\,$26$};
   \node[above] at (v1__4) {\footnotesize $14$,\,$24$};
   \node[below left] at (v0__4) {\footnotesize $13$,\,$23$};

\end{tikzpicture}
   \subcaption{Subdivision dual to caterpillar tree}
   \label{subfig:caterpillar}
   \end{subfigure}
  \caption{Two subdivisions of $\Delta(2,6)$ and their tropical linear spaces}
  \label{fig:corank}
\end{figure}

By Proposition~\ref{prop:mp-facets} the facets of any matroid polytope are either hypersimplex facets or induced by hypersimplex splits.
In the following we will be interested in the set of hypersimplex splits arising from the split flacets of a given matroid.
The next result explains what happens if that matroid is a split matroid.
\begin{proposition} \label{prop:prod_uniform}
  Let $M$ be a split $(d,n)$-matroid which is connected.
  Then the corank vector $\rho(M)$ is contained in the relative interior of a simplicial cone of $\Dr(d,n)$, and the dimension of that cone is given by the number of split flacets of $M$.
  In particular, $\rho(M)$ is a ray if and only if it induces a split of $\Delta(d,n)$.
  This is the case if and only if $M$ is a nested matroid with exactly three cyclic flats.
\end{proposition}
\begin{proof}
  Let $H$ be the set of hypersimplex splits corresponding to the split flacets of $M$.
  By definition the splits in $H$ are compatible.
  Since each subset of a compatible set of splits is again compatible it follows that the secondary cone spanned by $H$ is a simplicial cone.
  
  Recall that $M$ is nested if the cyclic flats form a chain. 
  The empty set and $[n]$ are two cyclic flats in any connected matroid.
  Assume that the matroid $M$ is nested with precisely three cyclic flats.
  Then the third cyclic flat $F$ induces the only split, since the restriction $M|F$ and the contraction $M\backslash F$ are uniform matroids.

  Conversely, if the matroid $M$ is split with a unique split flacet $F$, then obviously $\emptyset\subsetneq F\subsetneq [n]$.
  Each circuit $C$ of $M$ with fewer than $d+1$ elements leads to valid inequality of the polytope $P(M)$.
  This inequality separates $P(M)$ from those vertices of the hypersimplex with $x_i=1$ for $i\in C$.
  Hence, the only split flacet $F$ contains the circuit $C$.
  The restriction $M|F$ is a uniform matroid and thus $\rank(C) = \rank(F)$.
  We get that $F$ is the closure of $C$.
  Hence we may conclude that $M$ is nested.
\end{proof}
Our next result generalizes \cite[Thm.~3.6]{HerrmannJensenJoswigSturmfels:2009}, which settled the case $d=3$.
\begin{theorem}\label{thm:dim}
  For the dimension of the Dressian we have
  \[
  \frac{1}{n}\binom{n}{d}-1 \ \leq \ \dim\Dr(d,n) \ \leq \ \binom{n-2}{d-1}-1 \enspace .
  \]
\end{theorem}
\begin{proof}
  Speyer showed that the \emph{spread} of any matroid subdivision of the hypersimplex $\Delta(d,n)$, i.e., its number of maximal cells, does not exceed $\binom{n-2}{d-1}$ \cite[Thm.~3.1]{Speyer:2005}.
  The dimension of a secondary cone of a subdivision $\Sigma$ is the size of a maximal linearly independent family of coarsest subdivisions which are refined by $\Sigma$.
  As each (coarsest) subdivision has at least two maximal cells, the dimension of the secondary cone is at most the spread minus one.
  This follows from the fact that at least $k$ (linearly independent) rays are necessary in order to generate a cone of dimension $k$.
  It follows that $\dim\Dr(d,n) \leq \tbinom{n-2}{d-1}-1$.
  The lower bound is given by Knuth's construction of stable sets in $J(d,n)$; see Example~\ref{example:knuth}.
\end{proof}
This gives the following asymptotic estimates.
\begin{corollary}
  For fixed $d$ the dimension of the Dressian $\Dr(d, n)$ is of order $\Theta(n^{d-1})$.
  Further, the asymptotic dimension of the Dressian $\Dr(d, 2d)$ is bounded from below by $\Omega(4^d d^{-3/2})$ and bounded from above by $O(4^d d^{-1/2})$.
\end{corollary}
\begin{proof}
  For fixed $d$ the lower and the upper bound in Theorem~\ref{thm:dim} both grow as fast as $n^{d-1}$ asymptotically.
  Stirling's formula yields that the binomial coefficient $2d \choose d$ grows like ${2^{2d}}/{\sqrt{\pi d}}$.
  Specializing the bounds in Theorem~\ref{thm:dim} to $n=2d$ thus yields
  \[
  \Omega\left( \frac{2^{2d-1}}{d\sqrt{\pi d}} \right) \ \leq \ \dim\Dr(d,2d) \ \leq \  O\left( \frac{2^{2d-2}}{\sqrt{\pi (d-1)}} \right) \enspace .
  \]
  Now the lower and the upper bound differ by a multiplicative factor of
  \[
  \frac{d\sqrt{d}}{2\sqrt{d-1}} \enspace ,
  \]
  which tends to $d/2$ when $d$ goes to infinity.
\end{proof}

The following example shows that not all matroid subdivisions are induced by a corank function.
\begin{example}
  The matroid subdivision $\Sigma$ of the hypersimplex $\Delta(2,6)$ induced by the lifting vector $(3,2,1,0,0,2,1,0,0,2,1,1,2,2,3)$ is not a corank subdivision.
  We give a hint how this claim can be verified.
  This subdivision $\Sigma$ has exactly $4$ maximal cells, which come as two pairs of isomorphic cells.
  One can check that $\Sigma$ does not agree with the corank subdivision induced by any of these maximal cells.
  The lifting-vector is obtained from a metric caterpillar tree with six leaves and unit edge lengths, see Figure~\ref{subfig:caterpillar}.
  Notice that the subdivision $\Sigma$ is realizable by a tropical point configuration, while the corank subdivision induced by the snowflake matroid $\snow$ is not;
  see \cite{HerrmannJoswigSpeyer:2012}.
\end{example}

Tropical geometry studies the images under the valuation map of algebraic varieties over fields with a discrete valuation; see, e.g., \cite[Chapter~3]{MaclaganSturmfels:2015}.
Let $\KKt$ be the field of formal Puiseux series over an algebraically closed field $\KK$.
The valuation map $\val:\KKt\to\RR\cup\{\infty\}$ sends a Puiseux series to the exponent of the term of lowest order.
Each $d$-dimensional subspace in the vector space $\KKt^n$ can be written as the column span of a $d{\times}n$-matrix $A$.
The maximal minors of $A$ encode that subspace as a \emph{Pl\"ucker vector}, which is a point on the \emph{Grassmannian} $\Gr_{\KKt}(d,n)$, an algebraic variety over $\KKt$.
Tropicalizing the Pl\"ucker vector of $A$ yields a tropical Pl\"ucker vector, i.e., a point on the Dressian $\Dr(d,n)$.
In fact the set of all tropical Pl\"ucker vectors which arise in this way is the \emph{tropical Grassmannian} $\TGr_{\characteristic \KK}(d,n)$.
The latter is the tropical variety which comes about as the tropicalization of $\Gr_{\KKt}(d,n)$, and this is a $d(n-d)$-dimensional polyhedral fan, which is a proper subset of $\Dr(d,n)$ unless $d=2$ or $(d,n)=(3,6)$; see \cite{SpeyerSturmfels:2004} and \cite{Speyer:2005}.
The precise relationship between the fan structures of $\TGr_{\characteristic \KK}(d,n)$ and $\Dr(d,n)$ is a topic of ongoing research.
Since the Pl\"ucker ideal, which defines $\Gr_{\KKt}(d,n)$, is generated by polynomials with integer coefficients, the tropical variety  $\TGr_{\characteristic \KK}(d,n)$ only depends on the characteristic of the field $\KKt$, which agrees with the characteristic of $\KK$.
The tropical Pl\"ucker vectors that lie in the tropical Grassmannian are called \emph{realizable}.
We also say that such a tropical Pl\"ucker vector can be \emph{lifted} to an ordinary Pl\"ucker vector.
The following was stated in \cite[Example 4.5.4]{Speyer:2005}.
We indicate a short proof for the sake of completeness.
\begin{proposition}\label{prop:realizable}
  Let $\pi$ be a $(d,n)$-tropical Pl\"ucker vector which can be lifted to an ordinary Pl\"ucker vector over $\KKt$.
  Then the cells in the subdivision of $\Delta(d,n)$ induced by $\pi$ necessarily correspond to matroids which are realizable over $\KK$.
\end{proposition}
\begin{proof}
  By our assumption there exists an ordinary Pl\"ucker vector $p$ which valuates to $\pi$.
  We can pick a matrix $A\in\KKt^{d\times n}$ such that for each $d$-set $I$ of columns we have $\det A_I = p_I$.
  It follows that $\val(\det A_I) = \pi_I$.
  Note that the matrix $A$ is not unique.

  Let $M$ be the matroid corresponding to a cell.
  Up to a linear transformation we may assume that $\pi$ is non-negative, and we have $\pi_I= 0$ if and only if $I$ is a basis of $M$.
  We will show that $A$ can be chosen such that the valuation of each entry is non negative.

  We apply Gaussian elimination to the $n\geq d$ columns of $A$.
  This way the classical Pl\"ucker vector associated with $A$ is multiplied with a non-zero scalar.
  Thus the tropical Pl\"ucker vector $\pi$ is modified by adding a multiple of the all-ones vector.
  In each step, among the possible pivots pick one whose valuation is minimal.
  Let $\gamma$ be the product of all pivot elements, and let $c\,t^g$ for $c\neq 0$ be the term of lowest order.
  By construction $g=\val(c\,t^g)=\val(\gamma)$ is a lower bound for the valuations of the minors of $A$, which is actually attained.
  Since $\pi$ is non-negative and since $\pi_I= 0$ if $I$ is a basis we conclude that $g=0$.

  Including possibly trivial pivots with $1$ we obtain exactly $d$ pivots, one for each row of $A$.
  Multiplying each row with the inverse of the lowest order term of the corresponding pivot does not change $\pi$.
  The resulting matrix $A'$ is a realization with entries whose valuations are non-negative.
  Hence we can evaluate the matrix $A'\in\KKt^{d\times n}$ at $t=0$.
  This gives us the matrix $B\in\KK^{d\times n}$ with the constant terms of $A'$.
  The matrix $B$ realizes $M$ since $\det B_I = 0$ if and only if the lowest order term of $\det A'_I$ is constant in $t$. 
\end{proof}

Our next goal is to prove a characterization of matroid realizability in terms of tropical Pl\"ucker vectors.
In the proof we will use a standard construction from matroid theory which will also reappear further below.
The \emph{free extension} of the $(d,n)$-matroid $M$ by an element $f\not\in[n]$ is the $(d,n{+}1)$-matroid which arises from $M$ by adding $f$ to the ground set such that it is independent from each $(d{-}1)$-element subset of $[n]$.

\begin{theorem}\label{thm:realizable}
  Let $M$ be a $(d,n)$-matroid.
  The corank vector $\rho(M)$ can be lifted to an ordinary Pl\"ucker vector over $\KKt$ if and only if $M$ is realizable over $\KK$.
\end{theorem}
\begin{proof}
  Let $\rho(M)$ be realizable.
  Since $P(M)$ occurs as a cell in the matroid subdivision induced by $\rho(M)$ the matroid $M$ is realizable due to Proposition~\ref{prop:realizable}.

  Conversely, let us assume that the matroid $M$ is realizable and the matrix $B\in\KK^{d\times n}$ is a full rank realization.
  The matrix $B$ has only finitely many entries, and these generate some extension field $\LL$ of the prime field of $\KK$.
  The field $\LL$ may or may not be transcendental, but it is certainly not algebraically closed.
  Hence there exists an element $\alpha\in\KK - \LL$ which is algebraic over $\LL$ of degree at least $n$.
  The vector $B\cdot \transpose{(1,\alpha,\ldots,\alpha^{n-1})}$ is $\LL$-linearly independent of any $d-1$ columns of $B$.
  We infer that even the free extension of $M$ is realizable over $\KK$.
  After altogether $n$ free extensions we obtain a matrix $C\in\KK^{d\times n}$ such that the block column matrix $[B|C]$ is a realization of the $n$-fold free extension of $M$.
  We define $A:=B+t\cdot C$, which is a $d{\times}n$-matrix with coefficients in $\KKt$.

  For any $d$-subset $I$ of $[n]$ and for any subset $S\subseteq I$ we denote by $D(S)\in\KKt^{d\times n}$ the matrix whose $k$-th column is the $k$-th column of $B$ if $k\in S$ and $t$ times the $k$-th column of $C$ otherwise.
  Then
  \[
  \det A_I \ = \ \det(B_I+t\cdot C_I) \ = \ \sum_{S\subseteq I} \det D_I(S) \enspace .
  \]
  Further, by choice of $C$, we have $\det D_I(S)=0$ if and only if $S$ is a dependent set in $M$, and $\val (\det D_I(S)) = d-\size S$ if $S$ is independent.
  For a fixed set $S\subseteq I$ the Puiseux series $\det D_I(S)$ has a term $c(S) t^{g(S)}$ of lowest order, and we have $g(S)=\val (\det D_I(S)) = d-\size S$.
  The field $\KK$ is an $\LL$-vector space, and the set
  \[
  \SetOf{c(S)}{S \text{ independent subset of } I}
  \]
  of leading coefficients is linearly independent over $\LL$.
  This is why we obtain $\val \det A_I = d-\rank(I)$, i.e., cancellation does not occur.
  That is, the ordinary Pl\"ucker vector of the matrix $A$ tropicalizes to $\rho(M)$.
\end{proof}

\goodbreak
\section{Rays of the Dressian}\label{sec:rays}
\noindent
The purpose of this section is to describe a large class of tropical linear spaces, which are tropically rigid, i.e., they correspond to rays of the corresponding Dressian.
Before we can define a  special construction for matroids we first browse through a few standard concepts.

Let $M$ be a connected matroid of rank $d$ with $[n]$ as its set of elements.
The \emph{parallel extension} of $M$ at an element $e\in[n]$ by $s\not\in[n]$ is the $(d,n{+}1)$-matroid whose flats are either flats of $M$ which do not contain $e$ or sets of the form $F+s$, where $F$ is a flat containing~$e$.
Among all connected extensions the parallel extension is the one in which the shortest length of a circuit that contains the added element is minimal.
In fact, that length equals two.
Similarly, the free extension is characterized by the following property:
Any circuit that contains the added element has length $d+1$, and this is the maximal length of such a circuit.

In general a \emph{coextension} of $M$ is the dual of an extension applied to the dual matroid $M^*$.
That is, a coextension of a $(d,n)$-matroid is a $(d{+}1,n{+}1)$-matroid.
Finally, a \emph{series-extension} is a parallel coextension.
\begin{definition}
  The \emph{series-free lift} of $M$, denoted as $\sfLift M$, is the  matroid of rank $d+1$ with $n+2$ elements obtained as the series-extension of $M'$ at $f$ by $s$, where $M'$ is the free extension of $M$ by $f$.
\end{definition}
Note that $\sfLift M$ is connected as $M$ is connected.
In the sequel we want to show that the corank subdivision of $\sfLift M$ yields a ray of the Dressian $\Dr(d+1,n+2)$, whenever $M$ is a $(d,n)$-split matroid.
Let us first determine the rank function and the bases of $\sfLift M$.
We write $fs$ as shorthand for the two-element set $f+s=\{f,s\}$.
\begin{lemma}\label{lem:sfLift}
  The set $B$ of size $d+1$ is a basis in $\sfLift M$ if and only if one of the following conditions hold:
  \begin{enumerate}
  \item $fs\subseteq B$ and $\rank_M(B-fs) = d-1$, or
  \item $f\in B$ and $s\not\in B$ and $\rank_M(B - f) = d$, or
  \item $f\not\in B$ and $s\in B$ and $\rank_M(B - s) = d$.
  \end{enumerate}
  Further, the rank of $S\subseteq [n]+fs$ is given by
  \begin{equation}\label{eq:sfLift:rank}
    \rank_{\sfLift M}(S) \ = \ \min\bigl(\,\rank_M(S-fs)+\size(fs \cap S),\ d+1\,\bigr) \enspace .
  \end{equation}
  The split flacets of $\sfLift M$ are those of $M$ and additionally $[n]$, the ground set of $M$.
\end{lemma}
\begin{proof} 
   Clearly each basis in $\sfLift M$ contains at least $f$ or $s$.
   Conversely, any basis $B$ of $M$ extends to a basis of $\sfLift M$ with either $f$ or $s$.
   A circuit of the free extension $M'$ of $M$ by $f$ that contains $f$ has size $d+1$.
   Hence each circuit of $\sfLift M$ that contains $f$ and $s$ has length $d+2$.
   In particular, this implies that each independent set $B$ in $M$ of size $d-1$ together with $fs$ forms a basis of $\sfLift M$.
   Any set which is dependent over $M$ is also dependent over $\sfLift M$.

   The formula for the rank function is a direct consequence of the description of the bases.
   We see that there is no circuit of length at most $d$, that contains $f$, $s$ or both.
   Proposition~\ref{prop:cyclic_flat} says that there is no flacet that contains $f$ or $s$.
   Contracting the set $[n]$ in $\sfLift M$ yields the uniform matroid of rank $1$ on the two-element set $fs$, and this is connected.
   For $S$ a subset of $[n]+fs$ and any set $F\neq [n]$ that does not contain $fs$ we have
   \begin{align*}
      \rank_{ (\sfLift M) / F }(S) =& \rank_{\sfLift M}(S+F) - \rank_{\sfLift M}(F)\\ =& \min\{ \rank_M(S+F-fs)+\size(fs \cap S),\, d+1\} - \rank_M(F)\\ =& \min\{\rank_{M/F}(S-fs)+\size(fs \cap S),\, d-\rank_M(F)+1\}\\ =& \rank_{ \sfLift (M / F) }(S)
   \end{align*}
   The matroid $\sfLift (M/F) = (\sfLift M) / F$ is connected if and only if $M/F$ is connected.
   The restriction $\sfLift (M|F)$ coincides with $M|F$.
   Both the restriction and contraction on $F$ are connected in $M$ if and only if they are connected in $\sfLift M$.
   We conclude that the split flacets of $\sfLift M$ are precisely the ones in our claim.
\end{proof}

Our next goal is to describe the maximal cells of the corank subdivision induced by $\sfLift M$.
To this end we first define the matroid $\sfLift^* M$ as the free coextension of $M$ by $f$, followed by the parallel extension at $f$ by $s$.
We call $\sfLift^* M$ the \emph{parallel-cofree lift} of $M$.
This new construction is related to the series-free lift by the equality
\[
\sfLift^* M \ = \ \bigl(\,\sfLift(M^*)\,\bigr)^* \enspace .
\]
A direct computation shows that the rank function is given by
\begin{equation}\label{eq:rank-parallel-cofree}
\begin{split}
   \rank_{\sfLift^* M}(S) \ &= \ \min\bigl(\, \rank_{\sfLift M}(S)+\size(fs - S)-1,\,\size S\,\bigr) \\
  &= \  \min\bigl(\,\rank_M(S-fs)+1,\ \size S\,\bigr) \enspace .
\end{split}
\end{equation}

One maximal cell of the corank subdivision induced by $\sfLift M$ is obvious, namely the matroid polytope $P(\sfLift M)$.
This is the case as $M$, and thus also $\sfLift M$, is connected.
Here is another one.
\begin{lemma}\label{lem:sf_subdivision1}
  The corank subdivision of $\sfLift^* M$ coincides with the corank subdivision of $\sfLift M$.
  Hence the matroid polytope $P(\sfLift^* M)$ is a maximal cell of the corank subdivision of $\Delta(d+1,n+2)$ induced by $\sfLift M$.
  Further, the cells $P(\sfLift M)$ and $P(\sfLift^* M)$ intersect in a common cell of codimension one. 
\end{lemma}
\begin{proof} 
  Let $S$ be a subset of $[n]+fs$ of size $d+1$. 
  We have $\rank_M(S-fs)\leq d$.
  From \eqref{eq:rank-parallel-cofree} we deduce that $\rank_{\sfLift^* M}(S) = \rank_M(S-fs)+1 \leq d+1 = \size S$, while Lemma~\ref{lem:sfLift} gives $\rank_{\sfLift M}(S) = \rank_M(S-fs)+\size(fs \cap S) \leq \size(S-fs)+\size(fs \cap S)=d+1$. 
  Combining these two arrive at the equation $\rank_{\sfLift^* M}(S) = \rank_{\sfLift M}(S)-\size(fs\cap S)+1$.
  This implies
  \[
    \rho(\sfLift^* M) + 1 \ = \ \rho(\sfLift M) + x_f + x_s \enspace .
  \]
  As a consequence the corank subdivision of $\sfLift^* M$ coincides with the corank subdivision of $\sfLift M$.
  The common bases of the matroids $\sfLift M$ and $\sfLift^* M$ are the bases of the direct sum $M\oplus U_{1,fg}$.
  The corresponding matroid polytope yields the desired cell of codimension one.
\end{proof}

For each split flacet $F$ of $M$ we let $N_F$ be the connected $(d+1,n+2)$-matroid with elements $[n]+fs$ which has the following list of cyclic flats: $\emptyset$, $[n]-F$ of rank $d-\rank(F)$, $[n]-F+fs$ of rank $d+1-\rank(F)$ and $[n]+fs$ of rank $d+1$.

Note that these sets form a chain.
This chain has a rank $0$ element, the ranks are strictly increasing, and for each set the rank is less than the size.
Hence these sets form the cyclic flats of a matroid.
Its rank function is given by $\rank(S) = \min\SetOf{\rank(G)+\size(S-G)}{G \text{ is a cyclic flat}}$; see \cite{BoninMier:2008}.
Hence, the rank function of $N_F$ satisfies
\begin{equation}\label{eq:N_F}
  \rank_{N_F}(S) = \min \bigl\{ d+1, \size(S), \size(S\cap F)+d+1-\rank_M(F),  \size(S\cap (F+fs) )+d-\rank_M(F)  \bigr\} \enspace .
\end{equation}
This is a nested matroid with exactly two split flacets, namely $[n]-F$ and $[n]-F+fs$.
The corresponding hypersimplex splits are not compatible, i.e., $N_F$ is not a split matroid.
The following result compares the corank in $\sfLift M$ with the corank in $N_F$.
\begin{proposition}\label{prop:nested-ineq}
  For each split flacet $F$ of $M$ and any set $S\subseteq[n]+fs$ with $\size(S)=d+1$ we have
  \begin{equation} \label{eq:nested-cell:main}
     d+1-\rank_{\sfLift M}(S)+\rank_M(F)-\size(S\cap F)\ \geq \ d+1-\rank_{N_F}(S) \enspace .
  \end{equation}
\end{proposition}
\begin{proof}
  Since the size of $S$ equals $d+1$ the equation \eqref{eq:N_F} simplifies to
  \[
  d+1-\rank_{N_F}(S) \ = \ \max\bigl\{0, \rank_M(F)-\size(S\cap F), \rank_M(F)+1-\size(S\cap F)-\size(S\cap fs) \bigr\}
  \]
  if we subtract both sides from $d+1$.
  That expression is the corank of $S$ in the nested matroid $N_F$.
  This corank function gives the $(d+1,n+2)$-tropical Pl\"ucker vector $\rho(N_F)$.
  In the sequel we will make frequent use of the inequality
  \begin{equation}\label{eq:nested-cell:flacet}
     \rank_M(S-fs) \ \leq \ \rank(F)+\size(S-F-fs) \ = \ \size(S-fs)-\size(S\cap F)+\rank_M(F) \enspace ,
  \end{equation}
  which is a consequence of the fact that $F$ is a cyclic flat of $M$.

  To prove \eqref{eq:nested-cell:main} we distinguish three cases. 
  First, if neither $f$ nor $s$ are in $S$ the inequality \eqref{eq:nested-cell:main} is equivalent to
  \begin{equation}\label{eq:nested-cell:ineq}
    d+1-\rank_M(S)+\rank_M(F) \ \geq \ \max\{\size(S\cap F),\, \rank_M(F)+1 \}  \enspace ,
  \end{equation}
  as $\rank_{\sfLift M}(S) = \rank_M(S)< d+1$ by (\ref{eq:sfLift:rank}).
  The inequality (\ref{eq:nested-cell:ineq}) follows from $\rank_M(S)\leq d$ and \eqref{eq:nested-cell:flacet} with $\size(S-fs) = d+1$.
  Second, if $\size(fs\cap S)=1$, again by applying (\ref{eq:sfLift:rank}) the inequality~\eqref{eq:nested-cell:main} is equivalent to
  \[
  d-\rank_M(S-fs)+\rank_M(F) \ \geq \ \max\{\size(S\cap F),\, \rank_M(F) \} \enspace ,
  \]
  which holds due to the same arguments as in the first case with $\size(S-fs) = d$.
  Third, in the remaining case we have $s,f\in S$, which yields $\rank_M(S-fs)\leq \size(S-fs) = d-1$.
  This implies that the inequality~\eqref{eq:nested-cell:main} is equivalent to
  \begin{equation}\label{eq:nested-cell:third}
    d-1-\rank_M(S-fs) +\rank_M(F)-\size(S\cap F) \ \geq \ \max\{0,\, \rank_M(F)-\size(S\cap F) \} \enspace .
  \end{equation}
  If the maximum on the right hand side is attained at $\rank_M(F)-\size(S\cap F)$ that inequality holds trivially.
  We are left with the situation where the maximum on the right is attained solely by zero.
  This means that $\rank_M(F) < \size(S\cap F)$, which yields
  \begin{equation}\label{eq:nested-cell:third_final}
    d-\size(S\cap F)+\rank_M(F) \ \geq \ d-1 \ \geq \ \rank_M(S-fs) \enspace .
  \end{equation}
  If $\rank_M(S-fs)<d-1$ then \eqref{eq:nested-cell:third} is immediate.
  So we may assume that  $\rank_M(S-fs)=d-1$.
  From Lemma~\ref{lem:sfLift} we deduce that $S$ is a basis of $\sfLift M$.
  Since $F$ is also a flacet of $\sfLift M$ we get $\rank_M(F)\geq\size(S\cap F)$.
  However, this contradicts  $\rank_M(F) < \size(S\cap F)$, and we conclude that the case where the maximum to the right of \eqref{eq:nested-cell:third} cannot be attained at zero only.
  This final contradiction completes our proof.
\end{proof}

\begin{lemma}\label{lem:nested-cell}
  Let $M$ be a $(d,n)$-split matroid.
  Then for each split flacet $F$ of $M$ the matroid polytope $P(N_F)$ is a maximal cell of the corank subdivision of $\Delta(d+1,n+2)$ induced by $\sfLift M$.
  Further, the cell $P(N_F)$ shares a split flacet with $P(\sfLift M)$ and another one with $P(\sfLift^* M)$.
\end{lemma}
\begin{proof}
  We want to show that equality holds in \eqref{eq:nested-cell:main} if $S$ is a basis of $N_F$.
  In other words the corank lifting of $N_F$ agrees with the corank lift of $\sfLift M$ on $P(N_F)$, up to an affine transformation.
  Moreover, the bases of $N_F$ are lifted to height zero, while the lifting function is strictly positive on all other bases; see inequality~\eqref{eq:nested-cell:main}.
  This implies that $P(N_F)$ is a maximal cell in the corank subdivision of $\sfLift M$.

  The matroid $M$ is split, hence the contraction $M/F$ on the flacet $F$ is a uniform matroid of rank $d-\rank(F)$.
  Therefore, the rank function satisfies 
  \[
  \rank_M(S+F-fs)-\rank_M(F) \ = \ \min \{ \size(S-F-fs),\, d-\rank_M(F) \} \enspace .
  \]
  With Lemma~\ref{lem:sfLift} we get
  \begin{equation}\label{eq:rank_estimation}
  \begin{aligned}
    \rank_{\sfLift M}(S) \ &\leq \ \rank_{\sfLift M}(S+F) \\
    &= \ \min \{\rank_M(S+F-fs)+\size(S\cap fs),\, d+1\} \\
     &= \ \min \{\size(S-F)+\rank_M(F),\, d+\size(S\cap fs),\, d+1\} \}  \enspace .
  \end{aligned}
  \end{equation}
  The set $[n]-F$ is a flacet of rank $d-\rank_M(F)$ in $N_F$.
  For any basis $S$ of $N_F$ we get
  \[
  \rank_M(F)+1+\size(S-F) \ \leq \ d+1 \ = \ \size(S-F) + \size(S \cap F) \enspace .
  \]
  This implies that $\size(S \cap F) \geq \rank_M(F)+1$.
  Together with the inequality \eqref{eq:rank_estimation} we get 
  \[
  d+1-\rank_{\sfLift M}(S)-\size(S\cap F)+\rank(F) \ \leq \ 0  \enspace .
  \]
  This means that equality holds in \eqref{eq:nested-cell:main} whenever $S$ is a basis of $N_F$.

  As a consequence $P(N_F)$ is a maximal cell of the corank subdivision of $\sfLift M$.
  Clearly $P(N_F)$ intersects $P(\sfLift M)$ in a codimension-$1$-cell that is contained in
  \[
     P_{\sfLift M}(F)  \ = \ P_{N_F}([n]-F+fs) \enspace .
  \]
  By Lemma~\ref{lem:sf_subdivision1} the same kind of argument holds for $\sfLift^*M$.
   That is, $P(N_F)$ intersects $P(\sfLift M)$ in a codimension-$1$-cell that is contained in $P_{\sfLift^* M}(F+fs) = P_{N_F}([n]-F)$.
\end{proof}
From the above we know that, for a split matroid $M$, the matroid polytopes of $\sfLift M$, $\sfLift^* M$ and the nested matroid $N_F$ for each flacet of $M$ form maximal cells of the corank subdivision induced by $\sfLift M$.
The following result describes the corresponding tropical linear space completely.

\goodbreak

\begin{theorem}\label{thm:rays}
  Let $M$ be a connected $(d,n)$-split matroid.
  Then the corank vector $\rho(\sfLift M)$ is a ray in the Dressian $\Dr(d+1,n+2)$.
  Moreover, it can be lifted to an ordinary Pl\"ucker vector over $\KKt$ if and only if $M$ is realizable over $\KK$.
\end{theorem}
\begin{proof}
  Let $\Sigma$ be the matroid subdivision of $\Delta(d+1,n+2)$ induced by $\rho(\sfLift M)$.
  By Lemma~\ref{lem:corank-cell}, Lemma~\ref{lem:sf_subdivision1} and Lemma~\ref{lem:nested-cell} the matroid polytopes $P(\sfLift M)$, $P(\sfLift^* M)$ and $P(N_F)$, for each flacet of $M$, form maximal cells of $\Sigma$.
  Further, those results show that for each flacet of these three kinds of matroids there are precisely two maximal cells in that list which contain that flacet.
  Since the dual graph of $\Sigma$ is connected this shows that these are all the maximal cells of $\Sigma$.

  Moreover, for each flacet $F$ of $M$, the three maximal cells $P(\sfLift M)$, $P(\sfLift^* M)$ and $P(\sfLift N_F)$ form a triangle in the tropical linear space.
  It follows from \cite[Proposition 28]{HerrmannJoswigSpeyer:2012} that $\Sigma$ does not admit a non-trivial coarsening, i.e., $\rho(\sfLift M)$ is a ray of the secondary fan and thus of the Dressian.

  Finally, by Theorem~\ref{thm:realizable}, the tropical Pl\"ucker vector $\rho(\sfLift M)$ can be lifted to an ordinary Pl\"ucker vector over $\KKt$ if and only if $\sfLift M$ is realizable over $\KK$.
  As $\KK$ is algebraically closed a matroid is realizable over $\KK$ if and only if any free extension or any series extension is realizable.
\end{proof}
Another general construction for producing tropical Pl\"ucker vectors and thus tropical linear spaces arises from point configurations in tropical projective tori.
This has been investigated in \cite{HerrmannJoswigSpeyer:2012}, \cite{Rincon:2013} and \cite{FinkRincon:2015}.
In the latter reference the resulting tropical linear spaces are called \emph{Stiefel tropical linear spaces}.
These two constructions are not mutually exclusive; there are Stiefel type rays which also arise via Theorem~\ref{thm:rays}.
Complete descriptions of the Dressians $\Dr(3,n)$ are known for $n\leq 8$.
All their rays are of Stiefel type or they arise from connected matroids of rank two via Theorem~\ref{thm:rays}.

Via our method non-realizable matroids of rank three lead to interesting phenomena in rank four.
In particular, the following consequence of the above answers \cite[Question 36]{HerrmannJoswigSpeyer:2012}.
\begin{corollary}
  The Dressian $\Dr(d,n)$ contains rays which do not admit a realization in any characteristic for $d=4$ and $n\geq 11$ as well as for $d\geq 5$ and $n\geq 10$.
  There are rays of the Dressian $\Dr(4,9)$ that are not realizable in characteristic $2$ and others that are not realizable in any other characteristic.
\end{corollary}
\begin{proof}
  The non-Pappus $(3,9)$-matroid and the Vamos $(4,8)$-matroid are not realizable in any characteristic.
  Both are connected and paving and hence split.
  The construction in Theorem~\ref{thm:rays} leads to non realizable rays in $\Dr(4,11)$ and $\Dr(5,10)$.
  Each free extension or coextension of such a matroid is again connected and split.
  Thus we obtain non realizable rays in all higher Dressians.

  Applying Theorem~\ref{thm:rays} to the Fano and the non-Fano $(3,7)$-matroids we obtain two rays in $\Dr(4,9)$.
  The first one is realizable solely in characteristic $2$, whereas the other one is realizable in all other characteristics.
\end{proof}

\begin{figure}
  \input{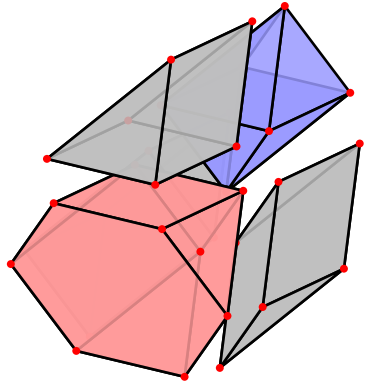}
  \caption{Projection of the corank subdivision of $\Delta(3,8)$ induced by \textcolor{blue}{$\sfLift \snow$} or, equivalently, induced by \textcolor{red}{$\sfLift^* \snow$}  There are \emph{five} maximal cells, one of which is almost entirely hidden in the picture}
  \label{fig:ray}
\end{figure}

\begin{example}
  Once again consider the snowflake matroid $\snow$ from Examples~\ref{example:snowflake-matroid} and \ref{example:snowflake-tree}.
  The corank vector of the series-free lift $\sfLift \snow$ is a ray in $\Dr(3,8)$.
  Since $\snow$ has three split flacets the corank subdivision has $3+2 = 5$ maximal cells.
  This is the, up to symmetry, unique ray of $\Dr(3,8)$ which does not arise from point configuration in the tropical projective $2$-torus; see \cite[Fig.~7]{HerrmannJoswigSpeyer:2012}.
  A projection of this subdivision to three dimensions is shown in Figure~\ref{fig:ray}.
\end{example}

\section{Concluding remarks and open questions}
\noindent
It would be interesting to characterize the split matroids in terms of their minors.
To this end we have the following contribution.
\begin{proposition} 
  The class of split matroids is closed under duality as well as under taking minors.
\end{proposition}
\begin{proof}
  The matroid polytope $P(M^*)$ of the dual $M^*$ of a $(d,n)$-matroid $M$ is the image of $P(M)\subset\RR^n$ under the the coordinate-wise transformation $x_i\mapsto 1-x_i$.
  In particular, $P(M^*)$ is affinely isomorphic with $P(M)$.
  In view of Proposition~\ref{prop:disconnected} we may assume that $M$ is connected.
  In this case any flacet $F$ of $M$ is mapped to the flacet $[n]-F$ of $M^*$.
  The compatibility relation among the splits is preserved under affine transformations.
  It follows that $M^*$ is split if and only if $M$ is.

  Assume that $M$ is a split matroid.
  Next we will show that the deletion $M|([n]-e)$ of an element $e\in [n]$ is again split.
  Since we already know that the class of split matroids is closed under duality it will follow that the class of split matroids is minor closed.

  Let $F$ be a split flacet of $M|([n]-e)$. 
  The $F$-hyperplane separates at least one vertex of $\Delta(d,[n]-e)$ from $P(M|([n]-e))$.
  This implies that the closure of $F$ in $M$ is a split flacet of $M$.
  For that closure there are two possibilities.
  So either $F$ or $F+e$ is a split flacet of $M$.

  Let us suppose that $F$ and $G$ are two split flacets of $M|([n]-e)$ which are incompatible.
  That is, there is some point $x$ in the relative interior of $\Delta(d,[n]-e)$ which lies on the $F$- and $G$-hyperplanes.
  We aim at finding at a contradiction by distinguishing four cases which arise from the two possibilities for the closures of the two flacets $F$ and $G$.

  First, suppose that $F$ and $G+e$ are split flacets of $M$.
  Then there exists some element $h\in G-F$, for otherwise $e$ would be in the closure of $F$ in $M$.
  For each $\epsilon > 0$ we define the vector $\hat x\in \RR^n$ with
  \begin{equation}\label{eq:minor:hatx}
    \hat x_e \, = \, \epsilon \;, \quad \hat x_h \, = \, x_h -\epsilon \quad \text{and} \quad \hat x_i \, = \, x_i \; \text{ for all other elements } i \enspace .
  \end{equation}
  If $\epsilon>0$ is sufficiently small then the vector $\hat x$ is contained in the relative interior of $\Delta(d,n)$.
  By construction $\hat x$ lies on the $F$- and $(G+e)$-hyperplanes, so that the corresponding splits are not compatible.
  This contradicts that $M$ is a split matroid.

  The second case where $F+e$ and $G$ are split flacets of $M$ is symmetric to the previous.

  Thirdly suppose that $F$ and $G$ are split flacets of $M$.
  Assume that $M|([n]-e)$ is connected.
  Then we have $\size(F\cap G)+d > \rank(F)+\rank(G)$ from Proposition~\ref{prop:compatible_splits}, and the same result implies that $F$ and $G$ are incompatible split flacets of $M$.
  Again this is a contradiction to $M$ being split.
  So we assume that $M|([n]-e)$ is disconnected.
  Then there exists an element $h\in [n]-F-G-e$, and we may construct a relatively interior point $\hat x\in \Delta(d,n)$ as in (\ref{eq:minor:hatx}).
  As before this leads to a contradiction to the assumption that $M$ is a split matroid.

  In the fourth and final case $F+e$ and $G+e$ are split flacets of $M$.
  As in the third case the desired contradiction arises from Proposition~\ref{prop:compatible_splits}, provided that $M|([n]-e)$ is connected.
  It remains to consider the situation where $M|([n]-e)$ is disconnected.
  Then we can find elements $f\in F-G$, $g\in G-F$ and $h\in [n]-F-G-e$.
  As a minor variation to (\ref{eq:minor:hatx}) we let
  \[
  \hat x_e \, = \, \epsilon \;, \quad \hat x_f \, = \, x_f -\epsilon \;, \quad  x_g \, = \, x_g -\epsilon \quad \text{and} \quad \hat x_i \, = \, x_i \; \text{ for all other elements } i \enspace .
  \]
  The vector $\hat x$ lies on the $(F+e)$- and $(G+e)$-hyperplanes, as well as in the relative interior of $\Delta(d,n)$.
  This entails that the flacets $F+e$ and $G+e$ are incompatible, and this concludes the proof.
\end{proof}
So it is natural to ask for the following.
\begin{question}
  What are the forbidden minors for the split matroids?
\end{question}
We want to list what we know about this question.
The only disconnected minimal excluded minor is the $(4,8)$-matroid in Example~\ref{example:direct_sum}.
One can show that the rank of a connected excluded minor must be at least $3$.
The class of split matroids is also closed under dualization.
Hence the number of elements is at least $6$.
There are precisely four excluded minors of rank $3$ on $6$ elements, up to symmetry.
One of them is the matroid in Example~\ref{example:non_split_matroid}, and a second one is its dual.
The third example is the nested matroid $\sfLift(\sfLift U_{1,2})$; see Example~\ref{example:non_split_nested}.
Finally, the fourth case has an extra split and is represented by the vectors:
$(1,0,0)$, $(1,0,0)$, $(0,1,0)$, $(1,1,0)$, $(0,0,1)$, $(1,0,1)$.

\smallskip

Here is another class of matroids of recent interest; see, e.g., Fife and Oxley \cite{FifeOxley:2017}.
A \emph{laminar} family $\cL$ of subsets of $[n]$ satisfies for all sets $A,B\in\cL$ either $A\cap B = \emptyset$, $A\subseteq B$ or $B\subseteq A$.
Furthermore, let $c$ be any real valued function on $\cL$, and this is called a \emph{capacity function}.
A set $I$ is an independent set of the \emph{laminar matroid} $L=L([n],\cL,c)$ if $\size(I\cap A) \leq c(A)$ for all $A\in\cL$.
Here the triplet $([n],\cL,c)$ is called a \emph{presentation} of $L$.
By \cite[Theorem 2.7]{FifeOxley:2017} each loop-free laminar matroid has a unique canonical presentation where the laminar family is the set of closures of the circuits, and the capacity function assigns to each set in the laminar family its rank.
The class of split matroids and the class of laminar matroids are not contained in one another:
The Fano matroid is a split matroid, but it is not laminar as it has closed circuits of size three which share exactly one element.
On the other hand the nested matroid from Example~\ref{example:non_split_nested} is not split.
However, each nested matroid is laminar \cite[Proposition 4.4]{FifeOxley:2017}.

\smallskip

It may be of general interest to look at tropical linear spaces where the matroidal cells correspond to matroids from a restricted class.
For instance, Speyer \cite{Speyer:2009} looks at series-parallel matroids, and he conjectures that the tropical linear spaces arising from them maximize the $f$-vector.
Tropical linear spaces all of whose maximal cells come from split matroids are necessarily one-dimensional, i.e., they are trees.
For instance, this is always the case for $d=2$.

\smallskip

Conceptually, it would be desirable to be able to write down all rays of all the Dressians and the tropical Grassmannians.
Due to the intricate nature of matroid combinatorics, however, it seems somehow unlikely that this can ever be done in an explicit way.
The next best thing is to come up with as many ray classes as possible.
In \cite{HerrmannJoswigSpeyer:2012} tropical point configurations are used as data, whereas here we look at split matroids and their corank subdivisions.
A third class of rays comes from the nested matroids.
However, their analysis is beyond the scope of the present paper.
It can be shown that the corank subdivision of a connected matroid $M$ is a ``$k$-split'' in the sense of Herrmann \cite{Herrmann:2011} if and only if $M$ is a nested matroid with $k+1$ cyclic flats.
The proof for this claim will be given elsewhere.

All known rays of the Dressians arise from corank vectors of various matroids.
So the following is another obvious challenge.
\begin{question}
  Is there a ray in any Dressian that does not induce a corank subdivision?
\end{question}

\smallskip

A \emph{polymatroid} is a polytope associated with a submodular function.
This generalizes matroids given by their rank functions.
Since splits are defined for arbitrary polytopes there is an obvious notion of a ``split polymatroid''.
It seems promising to investigate them.

\smallskip

Both polymatroids and tropical Pl\"ucker vectors are closely related to ``integral discrete functions'' which occur in discrete convex analysis; see, e.g., Murota \cite{Murota:2003}.
In that language a tropical Pl\"ucker vector is the same as an ``$M$-concave function'' on the vertices of the underling matroid polytope.
It would be interesting to investigate the notation of splits and realizability in terms of $M$-convexity.
Hirai took a first step in this direction in \cite{Hirai:2006}, where he studies splits of ``polyhedral convex functions''.

\section*{Acknowledgment}
\noindent
We are indebted to Hiroshi Hirai for asking about the relationship between split matroids and laminar matroids.
Further, we thank Amanda Cameron, Jorge Alberto Olarte and Raman Sanyal for various remarks and suggestions.
Finally, we are grateful to two anonymous referees for numerous very detailed comments which helped to improve the exposition.

\appendix

\section{Some Matroid Statistics}\label{app:statistics}
\noindent
Matsumoto, Moriyama, Imai and Bremner classified matroids of small rank with few elements \cite{Matsumoto:2012}.
A summary is given in Table~\ref{tab:matroid-classes} below.
Based on the census of \cite{Matsumoto:2012} we determined the percentages of paving and split matroids.
The results are given in Table~\ref{tab:matroid-special}.
That computation employed \polymake \cite{DMV:polymake}, and the results are accessible via the new database at \href{https://db.polymake.org/}{db.polymake.org}.
In all tables we marked entries with $-$ that have not been computed due to time and memory constraints.

Filtering all $190214$ matroids of rank $4$ on $9$ elements for paving, sparse paving and splits matroids took about $2000 \sec$ with \polymake version 3.1
(AMD Phenom II X6 1090T with 3.6 GHz single-threaded, running openSUSE 42.1).
We expect that the computation for all $(4,10)$-matroids, which is the next open case, would take much more than $600$ CPU days.

\begin{example}
  All matroids of rank $d$ on $d+2$ elements are split matroids.
  Table~\ref{subtab:paving} shows that most of these are not paving.
\end{example}

\begin{table}[th]
  \centering
  \caption{The number of isomorphism classes of all matroids of rank $d$ on $n$ elements, see \cite[Table 1]{Matsumoto:2012}}
  \label{tab:matroid-classes}
  \vspace{-0.3cm}
  \footnotesize
   \begin{tabular}{@{\hspace{0.3cm}}r@{\hspace{0.6cm}}rrrrrrrrr}
    \toprule
      \multicolumn{1}{l}{$d\backslash n$} & $4$ & $5$ & $6$ & $7$ & $8$ & $9$ & $10$ & $11$ & $12$\\
    \midrule
    2  &  7 &  13 &  23 &  37 &  58 &     87 &        128 & 183 & 259\\
    3  &  4 &  13 &  38 & 108 & 325 &   1275 &      10037 & 298491 & 31899134\\
    4  &  1 &   5 &  23 & 108 & 940 & 190214 & 4886380924 & $-$ & $-$\\
    5  &    &   1 &   6 &  37 & 325 & 190214 &        $-$ & $-$ & $-$\\
    6  &    &     &   1 &   7 &  58 &   1275 & 4886380924 & $-$ & $-$\\
    7  &    &     &     &   1 &   8 &     87 &      10037 & $-$ & $-$\\
    8  &    &     &     &     &   1 &      9 &        128 & 298491 & $-$\\
    9  &    &     &     &     &     &      1 &         10 &    183 & 31899134\\
    10 &    &     &     &     &     &        &          1 &     11 & 259\\
    11 &    &     &     &     &     &        &            &      1 & 12\\
    \bottomrule
  \end{tabular}
\end{table}

\begin{table}[th]
\caption{The percentage of paving and split matroids among the isomorphism classes of all matroids of rank $d$ on $n$ elements}
\label{tab:matroid-special}
\begin{subtable}[c]{0.48\textwidth}
  \subcaption{Paving matroids}
  \label{subtab:paving}
  \scriptsize
  \setlength{\tabcolsep}{1.5mm}
\begin{tabular}{@{\hspace{0.3cm}}r@{\hspace{0.2cm}}rrrrrrrrr}
      \toprule
      \multicolumn{1}{l}{$d\backslash n$} & $4$ & $5$ & $6$ & $7$ & $8$ & $9$ & $10$ & $11$ & $12$\\
      \midrule
         2  &  57 &  46 &  43 &  38 &  36 &  33 & 32 & 30 & 29\\
         3  &  50 &  31 &  24 &  21 &  21 &  30 & 52 & 78 & 91\\
         4           & 100 &  40 &  22 &  17 &  34 &  77 & $-$ & $-$ & $-$\\
         5                    &         & 100 &  33 &  14 &  12 &  63 & $-$ & $-$ & $-$\\
         6                    &         &         & 100 &  29 &  10 &  14 & $-$ & $-$ & $-$\\
         7                    &         &         &         & 100 &  25 &   7 &  17 & $-$ & $-$\\
         8                    &         &         &         &         & 100 &  22 &   5 &  19 & $-$\\
         9                    &         &         &         &         &         & 100 &  20 & 4 & 16\\
         10                   &         &         &         &         &         &          &     100 &      18 & 3\\
         11                   &         &         &         &         &         &          &         &      100 & 17\\
      \bottomrule
      \vspace{0.1cm}
  \end{tabular}
\end{subtable}
   \hspace{0.1cm}
\begin{subtable}[c]{0.48\textwidth}
\subcaption{Split matroids}
\scriptsize
\setlength{\tabcolsep}{1.5mm}
 \begin{tabular}{@{\hspace{0.3cm}}r@{\hspace{0.2cm}}rrrrrrrrr}
      \toprule
      \multicolumn{1}{l}{$d\backslash n$} & $4$ & $5$ & $6$ & $7$ & $8$ & $9$ & $10$ & $11$ & $12$\\    
      \midrule
         2  &  100 &  100 &  100 &  100 &  100 &  100 & 100 & 100 & 100\\
         3  &  100 &  100 &  89 &  75 &  60 &  52 & 61 & 80 & 91\\
         4           & 100 &  100 &  100 &  75 &  60 &  82 & $-$ & $-$ & $-$\\
         5                    &         & 100 &  100 &  100 &  60 &  82 & $-$ & $-$ & $-$\\
         6                    &         &         & 100 &  100 &  100 &  52 & $-$ & $-$ & $-$\\
         7                    &         &         &         & 100 &  100 &   100 &  61 & $-$ & $-$\\
         8                    &         &         &         &         & 100 &  100 &   100 &  80 & $-$\\
         9                    &         &         &         &         &         & 100 &  100 & 100 & 91\\
         10                   &         &         &         &         &         &          &     100 &      100 & 100\\
         11                   &         &         &         &         &         &          &         &      100 & 100\\
      \bottomrule
      \vspace{0.1cm}
  \end{tabular}   
\end{subtable}
\end{table}

\goodbreak

\bibliographystyle{alpha}
\bibliography{References}

\end{document}